\documentclass{amsproc}
\usepackage{amscd, enumerate}
\usepackage[initials]{amsrefs}
\newtheorem{theorem}{Theorem}[section]
\newtheorem{lemma}[theorem]{Lemma}
\newtheorem{proposition}[theorem]{Proposition}
\newtheorem{corollary}[theorem]{Corollary}
\newtheorem{Main Theorem}[theorem]{Main Theorem}
\theoremstyle{plain}
\newtheorem{Main}{Main Theorem}

\newtheorem{notation}[theorem]{Notation}
\theoremstyle{definition}
\newtheorem{definition}[theorem]{Definition}
\newtheorem{example}[theorem]{Example}

\theoremstyle{remark}
\newtheorem{remark}[theorem]{Remark}

\numberwithin{equation}{section}

\newcommand{\Def}{\stackrel{\mathrm{def}}{=\!\!=}}
\newcommand{\pf}{\noindent {\em Proof.\ }}
\newcommand{\ep}{
\hfill {$\square$} \medskip} 

\DeclareMathOperator{\charac}{char}
\DeclareMathOperator{\Hom}{Hom}
\DeclareMathOperator{\id}{id}

\font\smbfit=cmbxti10
\begin{document}

\title{Polynomial-like elements in vector spaces with
group actions}

\author[Minh Kha]{Minh Kha${}^{\natural}$}
\address{Department of Mathematics,
The University of Arizona, Tucson, Arizona, 85721, USA}
\email{minhkha@math.arizona.edu}
\thanks{${}^{\natural}$ The author acknowledges the support of the NSF Grant DMS-1517938}
\author[Vladimir Lin]{Vladimir Lin}
\address{Department of Mathematics, Technion-Israel Institute of Technology, Haifa, Israel 32000}
\email{vlalin@gmail.com}
\email{vlin@technion.ac.il}


\subjclass[2010]{Primary 54C40, 14E20; Secondary 46E25, 20C20}
\date{January 1, 1994 and, in revised form, June 22, 1994.}

\dedicatory{To the memory of Selim Grigorievich Krein, a great man and a great mathematician}

\keywords{Algebraic geometry, Group actions, Periodic differential operators}

\begin{abstract}
In this paper, we study polynomial-like elements in vector spaces equipped with group actions. We first define these elements via iterated difference operators. In the case of a full rank lattice acting on an Euclidean space, these polynomial-like elements are exactly polynomials with periodic coefficients, which are closely related to solutions of periodic differential equations. Our main theorem confirms that if the space of polynomial-like elements of degree zero is of finite dimension then for any $n \in \mathbb{Z}_+$, the space consisting of all polynomial-like elements of degree at most $n$ is also finite dimensional.
\end{abstract}

\maketitle

\section*{Introduction}


In 1984, T. Lyons and D. Sullivan \cite{LyoSul} used the
connections between the theory of harmonic functions and the theory of stochastic processes to prove that on a nilpotent covering of a compact Riemannian manifold, there are no nonconstant positive (and a fortiori no nonconstant bounded) harmonic functions.
In \cite {Lin}, a new approach was proposed, applicable both to bounded holomorphic functions on nilpotent coverings of complex spaces and to bounded harmonic functions on such coverings of Riemannian manifolds.

In the case of a compact base with a fixed Riemannian metric, the question naturally arises of the structure of spaces of holomorphic or harmonic functions on coverings. In particular, it is natural to expect that on nilpotent coverings the spaces of the corresponding functions of bounded polynomial growth are finite-dimensional. In complex-analytic case, some results for {\it abelian} coverings were obtained by A. Brudnyi \cite{BrudnyiA} and then, in both complex-analytic and harmonic cases in the paper of P. Kuchment and Y. Pinchover \cite{KuchPincho}. 

On the other hand,  in the series of three papers \cite{ColdMinic1}-\cite{ColdMinic3}, T. Colding and W. Minicozzi studied harmonic functions of restricted growth on Riemannian manifolds. In particular, it follows from their results that the spaces of harmonic or holomorphic functions of restricted polynomial growths on nilpotent coverings of {\em{K{\"{a}}hler}} manifolds are of finite dimension.

Another part of our motivation in studying polynomial-like elements via difference operators approach comes from the related studies of periodic equations, e.g., Liouville type results for elliptic equations of second-order with periodic coefficients (in divergence form) on Euclidean spaces, which
appeared in the work \cite{AveLin} of M. Avellaneda and F.-H. Lin. It is also worthwhile to note that analogous Liouville type results have been established in \cite{MosStru}. Such Liouville type results show that every solution with polynomial growth of the equation admits a representation: it is a linear combination of polynomials whose coefficients are periodic functions and moreover, each of these polynomials is also a solution. Hence, the spaces of such solutions (with a fixed polynomial growth) are finite dimensional.
As a first attempt to generalize some of these results, our first step is to give a definition of polynomial-like elements in vector spaces equipped with group actions. Rather than using explicit formulas, we choose to define in a more invariant way by using difference operators approach.

Let us give a brief outline of the paper.
Assume that $G$ is a group that acts linearly on a vector space $A$. Section \ref{sect: Polynomial-like elements in vector spaces with
group actions} is devoted to defining polynomial-like elements (or $G$-polynomials) in $A$ through the  iterated difference operators $D^n$ ($n \in \mathbb{Z}_+$). In Subsection \ref{subsect: Iterated difference operators and polynomial-like functions}, these iterated difference operators, which can be considered as analogs of the usual derivatives or the difference operators for $G$-moduli spaces, are introduced inductively via the action of $G$ on $A$. In Subsection \ref{subsect: Polynomial-like elements}, $G$-polynomials of degree at most $n$ in $A$ are defined as elements in the kernel of the $(n+1)^{th}$-iterated difference operator $D^{n+1}$.

Motivating from Liouville type results, we would like to understand the finite dimensionality of the spaces of polynomial-like elements under certain conditions. Our main aim is to prove the following theorem:
\begin{Main}\label{Main Theorem}
Let $F$ be a field of characteristic $0$,
$G$ be a group, and $A$ be a $F$-vector space endowed with
a linear right $G$-action. Let $A^G$ be the space consisting of all $G$-invariant elements in $A$ and $\mathcal{P}_n(G,A)$ be the space consisting of all $G$-polynomials of degree at most $n$ in $A$.
If the group $\widetilde G=G/[G,G]$ is finitely generated
and $\dim_F A^G<\infty$ then
\begin{equation*}\label{finiteness dimP}
\dim_F {\mathcal P}_n(G,A)<\infty\quad \text{for every} \ \, n\in\mathbb Z_+\,.
\end{equation*}
\end{Main}

We will develop the necessary tools, and then use them to prove this theorem in the rest of Section \ref{sect: Polynomial-like elements in vector spaces with group actions}, i.e., Subsections \ref{subsect: Polymorphisms}, \ref{subsect: Embeddings of polynomials to cochains with values in AG}, \ref{subsect: Coboundary operators and iterated difference operators}, and \ref{subsect: Iterated difference operators and polymorphisms}.

In Section \ref{sect: Polynomial-like elements in a ring}, we examine more properties of the iterated difference operators if $A$ has an additional ring structure that is compatible with the group action. In more details, we prove $(a)$ a Leibniz formula for the operator $D^1$, $(b)$ the set of all polynomial-like elements in $A$ is a subring, $(c)$ $D^n$ is a linear operator over the ring $A^G$. These 
results are needed for the next section.

In Section \ref{sect: Polynomial-like elements related to lattice}, we apply the above Main Theorem to the case when $G$ is a lattice acting naturally on $A$. We consider here an important example when $A$ is a $G$-invariant subspace of the algebra of continuous functions on the Euclidean space $\mathbb{R}^r$, where $r$ is the rank of the lattice $G$. In this example, we characterize completely $G$-polynomials (see Proposition \ref{Gamma-periodic polynomials are
polynomial-like elements}): these are exactly polynomials with $G$-invariant coefficients \footnote{These polynomials are called Floquet functions, which play an important role in studying the spectral theory of periodic differential operators (see e.g., \cite{KuchSurvey}).}, which we introduce at the beginning of Subsection \ref{subsec: G-periodic polynomials}.
Then in Subsection \ref{subsect: Gamma-periodic polynomials in Gamma-invariant subspaces}, we give a useful interpretation of the Main Theorem when $A$ is the space of all classical global solutions of a $G$-periodic linear differential operator $\mathcal{D}$. Finally, Subsection \ref{polynomial-like-solutions} provides some remarks related to periodic operators acting on co-compact regular Riemannian coverings whose deck transformation groups $G$ are not necessarily abelian.

\vspace{3pt}
\noindent
\textbf{Acknowledgments}: The authors are grateful to the referees for useful comments on this manuscript. The work of the first author was partially supported by the NSF grant DMS-1517938. Minh Kha expresses his gratitude to the NSF for the support.

\section{Polynomial-like elements}
\label{sect: Polynomial-like elements in vector spaces with
group actions}



\subsection{Iterated difference operators in $G$-moduli}
\label{subsect: Iterated difference operators and polynomial-like functions}
Let $G$ be a group with the unity $\mathbf e$ and $A$ be an additive abelian group.
\begin{definition}
\label{cochains}
For $n\in\mathbb{Z}_+$, let us denote by ${\mathcal C}^n(G,A)$ 
the additive group of all {\em normalized} $n$-cochains of $G$
with the values in $A$;
that is, ${\mathcal C}^0(G,A)=A$ and for $n\in\mathbb N$ the group
${\mathcal C}^n (G,A)$ consists of all functions
$$
c\colon\, G^n=\underset{n}{\underbrace{G\times\cdots\times G}}
\ni (g_1,...,g_n)\mapsto c(g_1,...,g_n)\in A
$$
such that $c(g_1,...,g_n)=0$ whenever at least one of the elements
$g_1,...,g_n$ equals $\mathbf e$.
\smallskip
\end{definition}

Suppose that $A$ is endowed with a
right $G$-module structure 
\begin{equation}\label{right G-action in A}
A\ni a\mapsto a^g\in A, \ \ g\in G\,.
\end{equation}
Such a structure induces the following right $G$-actions in
cochain groups ${\mathcal C}^n (G,A)$:
\begin{equation*}\label{right G-actions on cochains}
\aligned
\ & {\mathcal C}^n (G,A)\ni c\mapsto c^g\in {\mathcal C}^n (G,A)\,, \ \
c^g(g_1,...,g_n)=[c(g_1,...,g_n)]^g\,, \\
&\hskip7cm (g,g_1,...,g_n\in G\,, \ n\in\mathbb Z_+)\,.
\endaligned
\end{equation*}
These actions give rise to group homomorphisms
\begin{equation*}\label{iterated difference operators}
D^n\colon A \to {\mathcal C}^n (G,A) \ \ (n\in\mathbb Z_+)
\end{equation*}
defined as follows. First, we define homomorphisms
\begin{equation*}\label{domain and range of Delta homomorphisms}
d_n\colon\,{\mathcal C}^{n-1}(G,A)\to {\mathcal C}^n(G,A)
\ \ (n \ge 1)
\end{equation*}
by the formulas
\begin{equation}\label{d homomorphisms}
\aligned
(d_{n}c)(g_1,\ldots,g_{n-1},g_n)
&=c^{g_n}(g_1,\ldots,g_{n-1})-c(g_1,\ldots,g_{n-1}) \\
&=[c(g_1,\ldots,g_{n-1})]^{g_n}-c(g_1,\ldots,g_{n-1}) \\
&\qquad\quad (c\in{\mathcal C}^{n-1}(G,A)\,, \ \
g_1,\ldots,g_n\in G\,, \ \ n>1)\,.
\endaligned
\end{equation}
Using these homomorphisms, we define homomorphisms
$D^n\colon A\to\mathcal C^n(G,A)$ by the recursion relations
\begin{equation}\label{definition of iterated difference operators}
\aligned
\ & D^0=\textbf{id}_A \ \text{(the identity operator in} \ A) \ \, \text{and} \ \,
D^n = d_n D^{n-1} \ \ \text{for} \ n\in\mathbb N\,, \\
&\text{or, which is the same,} \ \,
D^n = d_n \cdots d_1 D^0 \ \,
\text{for all} \ \, n\in\mathbb{Z}_+.
\endaligned
\end{equation}
These homomorphisms $D^n$ are called the {\em iterated difference operators}.

\begin{notation}\label{Not: Notation for products}
Let $n,s,i_1,...,i_s\in{\mathbb N}$, where $1\le s\le n$
and $1\le i_1<...<i_s\le n$. For any $g_1,...,g_n\in G$,
set
$$
\pi_{i_1,...,i_s} (g_1,...,g_n)
= g_1\cdot\ldots\cdot\widehat{g_{i_1}}
\cdot\ldots\cdot\widehat{g_{i_2}}
\cdot\ldots\cdot\widehat{g_{i_s}}\cdot\ldots\cdot g_n
$$
$($terms with hats in the right hand side must be omitted,
and the empty product is defined to be $\mathbf e$, the unity of $G$$)$.
\end{notation}

\begin{lemma}\label{Lm: formula for Dna}
For any $a\in A$, $n\in{\mathbb N}$, and $g_1,...,g_n\in G$, we have
\begin{equation}\label{formula for Dna}
\aligned
\lbrack D^n a\rbrack (g_1,...,g_n)
&=a^{g_1\cdots g_n}+\\
&+\sum_{s=1}^{n-1}(-1)^s\sum_{1\le i_1<...<i_s\le n}
a^{\pi_{i_1,...,i_s} (g_1,...,g_n)}
+(-1)^n a\,.
\endaligned
\end{equation}
\end{lemma}

\begin{proof}
The proof is by induction in $n$.
\end{proof}

\begin{remark}\label{Rmk: symmetry}
Let us define
\begin{equation}\label{Delta na and a(g1,...,gn)}
\aligned
\lbrack \Delta^n a \rbrack (g_1,...,g_n)
&:=\sum_{s=1}^n (-1)^s S_s(g_1,...,g_n),\ \,\text{where}\ \,
\\
S_s(g_1,...,g_n)
&:=\sum_{1\le i_1<...<i_s\le n} a^{\pi_{i_1,...,i_s} (g_1,...,g_n)}\,.
\endaligned
\end{equation}
Then formula (\ref{formula for Dna}) can be written in the form
\begin{equation}\label{Dna via Delta na and a(g1,...,gn)}
\lbrack D^n a \rbrack(g_1,...,g_n)
=a^{g_1\cdots g_n}+[\Delta^n a](g_1,...,g_n)\,,
\end{equation}

Each sum $S_s(g_1,...,g_n)$ in (\ref{Delta na and a(g1,...,gn)})
is a symmetric function of $g_1,...,g_n$, that is,
$$
S_s(g_{j_1},...,g_{j_n})=S_s(g_1,...,g_n)
$$
for any $g_1,...,g_n\in G$ and any permutation $(j_1,...,j_n)$
of the indices $1,...,n$. Therefore, {\sl the function
$[\Delta^n a](g_1,...,g_n)=[D^n a](g_1,...,g_n)-a^{g_1\cdots g_n}$
is symmetric in $g_1,...,g_n$.}
\hfill $\bigcirc$
\end{remark}

For any {\em abelian} $G$, the function
$(g_1,...,g_n)\mapsto a^{g_1\cdots g_n}$ is symmetric,
and  the above remark 
implies the symmetry of $[D^n a](g_1,...,g_n)$.
For non-abelian $G$ this may be wrong; however, the following
property holds in general case.

\begin{lemma}\label{Lm: D^n a is symmetric for a in Pn}
Suppose that $D^{n+1} a=0$.
Then, for any natural $k\le n+1$, the functions $a^{g_1\cdots g_k}$ and
$[D^k a](g_1,...,g_k)$ are symmetric in $g_1,...,g_k\in G$.
\end{lemma}

\begin{proof}
The assumption $D^{n+1} a=0$ implies that
$$
\aligned
a^{g_1\cdots g_{n+1}}
&=[D^{n+1} a](g_1,...,g_{n+1}) - [\Delta^{n+1} a](g_1,...,g_{n+1})\\
&=- [\Delta^{n+1} a](g_1,...,g_{n+1})\,.
\endaligned
$$
By Remark \ref{Rmk: symmetry}, $[\Delta^{n+1} a](g_1,...,g_{n+1})$
is symmetric in $g_1,...,g_{n+1}$; therefore,
$a^{g_1\cdots g_{n+1}}$ is symmetric as well.
For any $k\le n+1$, we have
$$
\displaystyle a^{g_1\cdots g_k}=\left. a^{g_1\cdots g_{n+1}}
\right |_{g_{k+1}=...=g_{n+1}=\mathbf e}\,;
$$
hence $a^{g_1\cdots g_k}$ is symmetric in $g_1,\ldots ,g_k$
and, by (\ref{Dna via Delta na and a(g1,...,gn)}),
$[D^k a](g_1,...,g_k)$ is symmetric.
\end{proof}

\begin{lemma}\label{Lm: [D{n+1}a](g1,...,g{n+1})=
[Dnag1](g2,...,g{n+1}) - [Dna](g2,...,g{n+1})}
Let $n\in\mathbb Z_+$ and $a\in A$. Then
\begin{equation}\label{[D{n+1}a](g1,...,g{n+1})=
[Dnag1](g2,...,g{n+1}) - [Dna](g2,...,g{n+1})}
\begin{split}
[D^{n+1} a](g_1,g_2,...,g_{n+1})=&
[D^n (a^{g_1})](g_2,...,g_{n+1}) - [D^n a](g_2,...,g_{n+1})\\
&\hskip60pt \text{\rm for all} \ \ g_1,g_2,...,g_{n+1}\in G\,.
\end{split}
\end{equation}
\end{lemma}

\begin{proof}
The proof is by induction in $n$. For $n=0,1$, formula
(\ref{[D{n+1}a](g1,...,g{n+1})=
[Dnag1](g2,...,g{n+1}) - [Dna](g2,...,g{n+1})})
looks as follows:
\begin{eqnarray}
\ & n=0:&\quad [D^1 a](g_1)=a^{g_1} - a\,;
\label{[D1a](g1) via a and ag1}\\
  & n=1:&\quad [D^{2} a](g_1,g_2)
                      =[D^1 (a^{g_1})](g_2) - [D^1 a](g_2)\,.
  \label{[D2 a](g1,g2) via a and ag1}
\end{eqnarray}
Formula (\ref{[D1a](g1) via a and ag1}) is just the definition
of $D^1$. Furthermore, we have:
$$
\aligned
\lbrack D^2 a\rbrack (g_1,g_2)&=\lbrack a^{g_1}-a \rbrack^{g_2}-\lbrack a^{g_1}-a\rbrack=a^{g_1g_2}-a^{g_2}-a^{g_1}+a
\\
&=[(a^{g_1})^{g_2}-a^{g_1}]-(a^{g_2}-a)
=[D^1 (a^{g_1})](g_2) - [D^1 a](g_2)\,,
\endaligned
$$
which coincides with (\ref{[D2 a](g1,g2) via a and ag1}).
This provides us with the base of induction.
\smallskip

\noindent Suppose now that for some $m\ge{2}$, the formula
(\ref{[D{n+1}a](g1,...,g{n+1})=
[Dnag1](g2,...,g{n+1}) - [Dna](g2,...,g{n+1})}) is fulfilled
for all $n\le m-1$. Let us show that it holds true for
$n=m$. According to (\ref{definition of iterated difference operators}),
$$
\aligned
\lbrack D^{m+1}a\rbrack (g_1,g_2,...,g_{m+1})
&=\{[D^m a](g_1,g_2,...,g_m)\}^{g_{m+1}} - [D^m a](g_1,g_2,...,g_m)\,,\\
[D^m a](g_2,...,g_{m+1})\hskip0.8cm
&=\{[D^{m-1} a](g_2,...,g_m)\}^{g_{m+1}} - [D^{m-1} a](g_2,...,g_m)\,,
\endaligned
$$
which implies
\begin{equation}\label{intermediate formula}
\aligned
& [D^{m+1}a](g_1,g_2,...,g_{m+1})+[D^m a](g_2,...,g_{m+1})\\
&\hskip40pt
=\{[D^m a](g_1,g_2,...,g_m)\}^{g_{m+1}} - [D^m a](g_1,g_2,...,g_m)\\
&\hskip80pt
+\{[D^{m-1} a](g_2,...,g_m)\}^{g_{m+1}} - [D^{m-1} a](g_2,...,g_m)\\
&\hskip40pt=\{[D^m a](g_1,g_2,...,g_m)+[D^{m-1} a](g_2,...,g_m)\}^{g_{m+1}} \\
&\hskip80pt
- \{[D^m a](g_1,g_2,...,g_m) + [D^{m-1} a](g_2,...,g_m)\}\,.
\endaligned
\end{equation}
By the induction hypothesis,
\begin{equation}\label{induction hypothesis}
[D^m a](g_1,g_2,...,g_m)
=[D^{m-1}(a^{g_1})](g_2,...,g_m)-[D^{m-1} a](g_2,...,g_m)\,.
\end{equation}
It follows from (\ref{intermediate formula}),
(\ref{induction hypothesis}), and (\ref{formula for Dna}) that
$$
\aligned
&[D^{m+1}a](g_1,g_2,...,g_{m+1})+[D^m a](g_2,...,g_{m+1})\\
&\hskip25pt=\{[D^m a](g_1,...,g_m)+[D^{m-1} a](g_2,...,g_m)\}^{g_{m+1}}\\
&\hskip55pt - \{[D^m a](g_1,...,g_m) + [D^{m-1} a](g_2,...,g_m)\}\\
=&\{[D^{m-1}(a^{g_1})](g_2,...,g_m)-[D^{m-1} a](g_2,...,g_m)
+[D^{m-1} a](g_2,...,g_m)\}^{g_{m+1}}\\
&\hskip5pt - \{[D^{m-1}(a^{g_1})](g_2,...,g_m)-[D^{m-1} a](g_2,...,g_m)
+ [D^{m-1} a](g_2,...,g_m)]\}\\
=&\{[D^{m-1}(a^{g_1})](g_2,...,g_m)\}^{g_{m+1}}
            -[D^{m-1}(a^{g_1})](g_2,...,g_m)\\
=& [D^m(a^{g_1})](g_2,...,g_{m+1})\,,
\endaligned
$$
which yields
$$
[D^{m+1}a](g_1,g_2,...,g_{m+1})
=[D^m(a^{g_1})](g_2,...,g_{m+1})-[D^m a](g_2,...,g_{m+1}).
$$
This completes the proof.
\end{proof}

\subsection{Polynomial-like elements}
\label{subsect: Polynomial-like elements}
\noindent From now on, whenever $A$ is a vector space
we assume that the given right $G$-action in $A$ is
{\em linear}, i.e., all the mappings (\ref{right G-action in A})
are linear operators. Then the homomorphisms $d_m$ and $D^n$
defined in (\ref{d homomorphisms}) and
(\ref{definition of iterated difference operators})
respectively, are linear operators.
\begin{definition}\label{Def: G-polynomials in A}
Let us denote by ${\mathcal P}_n={\mathcal P}_n(G,A)$ the kernel of the
homomorphism $D^{n+1}\colon\, A\to{\mathcal C}^{n+1}(G,A)$, i.e.,
\begin{equation*}\label{kernel of Dn}
\aligned
{\mathcal P}_n &={\mathcal P}_n(G,A)=\{a\in A\mid D^{n+1} a = 0\}\\
&=\{a\in A\mid [D^{n+1} a](g_1,...,g_{n+1}) = 0 \ \
\forall \, g_1,...,g_{n+1}\in G\} \ \ (n\in\mathbb Z_+)\,.
\endaligned
\end{equation*}
Clearly, ${\mathcal P}_n(G,A)$ is a subgroup of $A$
(a vector subspace whenever $A$ is a vector space with a linear $G$-action).
\smallskip

\noindent An element $p\in{\mathcal P}_n$ is said to be a
{\em $G$-polynomial in $A$ of order at most $n$},
or a {\em polynomial-like element} (of order at most $n$) in $A$.
It follows from (\ref{definition of iterated difference operators})
that
\begin{equation}\label{Pn is contained in P{n+1}}
{\mathcal P}_0\subseteq {\mathcal P}_1\subseteq\ldots
\subseteq {\mathcal P}_n\subseteq {\mathcal P}_{n+1}\subseteq \ldots\,.
\end{equation}
\end{definition}

\noindent Let $A^G$ denote the subgroup of $A$ consisting of all
$G$-invariant elements, i.e.,
\begin{equation*}\label{G-invariant elements}
A^G = \{a\in A\mid a^g=a, \ \ \forall g\in G\}\,.
\end{equation*}
If $A$ is a vector space and the given $G$-action in $A$ is linear
then $A^G$ is a vector subspace of $A$.
\smallskip

\begin{example}\label{Expl: polynomial-like elements in function spaces}
Let $G$ be an additive subgroup of $\mathbb R^r$ acting by translations
in the space $C(\mathbb R^r)$ of all (real or complex)
continuous functions on $\mathbb R^r$:
\begin{equation*}\label{translations}
f^g(x)=f(x+g) \quad\text{for all}\ \,
f\in C(\mathbb R^r)\,, \ x\in\mathbb R^r\,, \  g\in G\subseteq\mathbb R^r
\end{equation*}
(all the mappings $f\mapsto f^g$ are, in fact, automorphisms
of the algebra $C(\mathbb R^r)$).
Let $A\subseteq C(\mathbb R^r)$ be a
$G$-invariant subspace of $C(\mathbb R^r)$.
For instance, the following natural cases seem interesting:
\smallskip

\begin{itemize}

\item [$(a)$] $G=\mathbb R^r$ and $A=C(\mathbb R^r)$;

\item [$(b)$] $G\cong \mathbb Z^r$ is a lattice of rank $r$ in $\mathbb R^r$
and $A=C(\mathbb R^r)$;

\item [$(c)$] $G$ is as in $(b)$ and $A=\mathcal H(\mathbb R^r)$ is the space of all
harmonic functions on $\mathbb R^r$;

\item [$(d)$] $r=2m$ is even, $\mathbb R^r=\mathbb C^m$,
$G\cong \mathbb Z^{2m}$ is a lattice of rank $2m$ in $\mathbb C^m$,
and $A=\mathcal O(\mathbb C^m)$ is the space of all holomorphic functions
on $\mathbb C^m$.
\end{itemize}
\medskip

\noindent In case $(a)$ the space $A^G$ of all $G$-invariant elements of $A$
is just the field of constants (that is, either $A^G=\mathbb R$ or $A^G=\mathbb C$).
In cases $(b)-(d)$, $A^G$ is the space of all $G$-invariant functions
that are, respectively, $(b)$ continuous,
$(c)$ harmonic, $(d)$ holomorphic.
\smallskip

\noindent In fact, with help of
Proposition \ref{Prp: DnPn(G,AG) is contained in LnS(G,AG)}
proven in Subsection \ref{subsect: Iterated difference operators and polymorphisms}
below, one can see that in cases $(a)-(c)$
the space $\mathcal P_n(G,A)$ coincides with the space $A^G_n[x_1,...,x_r]$
of all polynomials in the standard coordinates
$x_1,...,x_r$ of degree at most $n$
with coefficients in $A^G$;
in case $(d)$ the space $A^G_n[z_1,...,z_m]$ of
polynomials in {\em complex} coordinates $z_1,...,z_m$
should be taken instead.
\end{example}

\begin{example}\label{Expl: polynomial-like meromorphic functions}
Let $G\subset\mathbb C$ be a lattice of rank $2$
and ${\mathcal M}={\mathcal M}({\mathbb C})$ be the field of meromorphic functions
of one complex variable $z\in\mathbb C$.
Then ${\mathcal P}_0(G,{\mathcal M})={\mathcal M}^G$ is just the field of elliptic
functions associated with the lattice $G$. It can be proven
that ${\mathcal P}_n(G,{\mathcal M})={\mathcal M}^G_n[z]$,
where ${\mathcal M}^G_n[z]$ stands for the space of all
polynomial in $z$ of degree at most $n$ with coefficients in the
field of elliptic functions ${\mathcal M}^G$.
\end{example}

\begin{example}\label{Expl: multiplicative
polynomial-like meromorphic functions}
It is interesting to take $A={\mathcal M}^*={\mathcal M}\setminus\{0\}$,
the multiplicative group of the field ${\mathcal M}$,
with a lattice $G\subset\mathbb C$ of rank $2$ acting by translations. Clearly,
${\mathcal P}_0(G,{\mathcal M}^*)=({\mathcal M}^*)^G$, the multiplicative group
of the field ${\mathcal M}^G$ of $G$-elliptic functions.

\end{example}

\begin{proposition}\label{Pn is G-invariant}
Each ${\mathcal P_n}(G,A)\subseteq A$ is a $G$-invariant subgroup
$($respectively, a $G$-invariant vector subspace whenever $A$ is
a vector space$)$.
\end{proposition}

\begin{proof}
Since ${\mathcal P_n}(G,A)$ is the kernel of a homomorphism (respectively,
linear operator), we only need to check that
$a^g\in{\mathcal P_n}(G,A)$ whenever $a\in{\mathcal P_n}(G,A)$ and $g\in G$.
In view of Lemma \ref{Lm: [D{n+1}a](g1,...,g{n+1})=
[Dnag1](g2,...,g{n+1}) - [Dna](g2,...,g{n+1})} (with $n+2$ instead of $n$),
the assumption $D^{n+1} a=0$ implies
$$
[D^{n+1}(a^g)](g_1,...,g_{n+1})=
[D^{n+1}a](g_1,...,g_{n+1})+[D^{n+2} a](g,g_1,...,g_{n+1})=0+0=0\,,
$$
which concludes the proof.
\end{proof}

\begin{notation}\label{Not: commutator subgroup}
$[G,G]$ denotes the {\em commutator subgroup} of a group $G$ and
$\widetilde G=G/[G,G]$ is the {\bf abelianization} of $G$.
For every $g\in G$, let $\widetilde g$ denote the image of $g$ in
$\widetilde G$ under the natural epimorphism $G\to\widetilde G$.
\end{notation}

\noindent Our main aim in the remaining of this Section is to achieve the following theorem:

\begin{theorem}\label{Main Theorem}
Let $F$ be a field of characteristic $0$,
$G$ be a group, and $A$ be a $F$-vector space endowed with
a linear right $G$-action.
If $\widetilde G$ is finitely generated
and $\dim_F A^G<\infty$ then 
$\dim_F {\mathcal P}_n(G,A)<\infty\quad \text{for every} \ \, n\in\mathbb Z_+\,.$
\end{theorem}

\noindent To prove this theorem,
we show that the quotient spaces
${\mathcal P}_n(G,A)/{\mathcal P}_{n-1}(G,A)$ \ ($n\in\mathbb N$)
may be embedded into an appropriated vector space of
$A^G$-valued {\em symmetric polylinear forms} on the group $G$.
Under certain conditions,
the latter space is of finite dimension, which makes it possible
to use the induction in $n$. To follow this idea,
we need to define certain spaces of polylinear forms
on $G$ and establish some relations between the
iterated difference operators $D^n$ and the coboundary operators $\delta$
related to the {\em trivial} left $G$ action in $A$.
We will do this in the next Subsections \ref{subsect: Polymorphisms}, \ref{subsect: Embeddings of polynomials to cochains with values in AG}, \ref{subsect: Coboundary operators and iterated difference operators} and \ref{subsect: Iterated difference operators and polymorphisms}.
\subsection{Group polymorphisms}
\label{subsect: Polymorphisms}
In this Subsection, we introduce the concept of polylinear forms on the group $G$ and its related properties.
\begin{definition}\label{Def: group polymorphism}
Let $G$ be a group and $B$ be an abelian group.
A $B$-valued {\em polymorphism}, or,
more precisely, {\em $n$-morphism} of $G$
is a function
$$L\colon\, \underset{n}{\underbrace{G\times\cdots\times G}}\to B$$
such that for each $i\in \{1,...,n\}$, the condition
\begin{equation}\label{polymorphism}
\aligned
L(g_1,...,&g_{i-1},g'g'',g_{i+1},...,g_n)\\
&=L(g_1,...,g_{i-1},g',g_{i+1},...,g_n)
+L(g_1,...,g_{i-1},g'',g_{i+1},...,g_n)
\endaligned
\end{equation}
is fulfilled for every $g_1,...,g_{i-1},g',g'',g_{i+1},...,g_n\in G$.
\smallskip

\noindent The set $\mathcal L_n(G,B)$ of all $B$-valued
$n$-morphisms of $G$ is an abelian group.
If $B$ is a vector space over a field $F$
then ${\mathcal L}_n(G,B)$ is a vector space over $F$ as well;
in this case polymorphisms $L\in {\mathcal L}_n(G,B)$
are also said to be $B$-valued {\em polylinear}
(or $n$-{\em linear}) {\em forms} on $G$.
\smallskip

\noindent It is convenient to set $\mathcal L_0(G,B)=B$.
\end{definition}

\noindent We skip the proof of the following quite elementary lemma, leaving this as an exercise for the reader.

\begin{lemma}\label{Lm: polymorphisms of Zr}
$(a)$ Let $G$ be a free abelian group of rank $r$ and
\begin{equation}\label{polymorphism of Zr}
L\colon\,
\underset{n}{\underbrace{G\times\cdots\times G}}\to B
\end{equation}
be a $n$-morphism of $G$ to an abelian group $B$.
Take any free basis $\mathbf h_1,...,\mathbf h_r$ of $G$ and set
\begin{equation}\label{the values of L on (hi1,...,hin)}
b_{i_1,...,i_n}=L(\mathbf h_{i_1},...,\mathbf h_{i_n})\,, \ \
1\le i_1,...,i_n\le r\,.
\end{equation}
Then for any $g_1,...,g_n\in G$ we have
\begin{equation*}\label{L(g1,...,gn) via L(hi1,...,hin)}
\aligned
L(g_1,...,g_n)
&=\sum_{1\le i_1,...,i_n\le r} g_{1,i_1}\cdots g_{n,i_n}\,
\cdot L(\mathbf h_{i_1},...,\mathbf h_{i_n})\\
&=\sum_{1\le i_1,...,i_n\le r} g_{1,i_1}\cdots g_{n,i_n}\, \cdot b_{i_1,...,i_n}\,,
\endaligned
\end{equation*}
where $g_{k,1},...,g_{k,r}\in{\mathbb Z}$ are the coordinates of the element
$g_k\in G$ represented in the basis $\mathbf h_1,...,\mathbf h_r$ \ $(k=1,...,n)$.
\smallskip

\noindent $(b)$ For any choice of $r^n$ elements
$b_{i_1,...,i_n}\in B\ \ (1\le i_1,...,i_n\le r)$,
there exists a unique $n$-morphism {\rm (\ref{polymorphism of Zr})}
satisfying {\rm (\ref{the values of L on (hi1,...,hin)})}.
\smallskip

\noindent $(c)$ Let $G\subset\mathbb R^r$ be a lattice of rank $r$,
$\mathbf h_1,...,\mathbf h_r$ be a basis of $G$, and
$L$ be a $n$-morphism of $G$
to a vector space $B$. Then there exists a unique
$B$-valued $n$-linear form
\begin{equation*}\label{B-valued n-linear form defined by n-morphism}
{\frak L}\colon\,
\underset{n}{\underbrace{\mathbb R^r\times\cdots\times\mathbb R^r}}\to B
\end{equation*}
that extends $L$ from $G$ to $\mathbb R^r$ $($i. e., $\frak L$ satisfies
${\frak L}|_G=L$$)$. This form $\frak L$ is defined by
\begin{equation*}\label{extention of L to Rr}
\aligned
{\frak L}(v_1,...,v_n)
&=\sum_{1\le i_1,...,i_n\le r} v_{1,i_1}\cdots v_{n,i_n}\,
L(\mathbf h_{i_1},...,\mathbf h_{i_n})\\
&=\sum_{1\le i_1,...,i_n\le r}
v_{1,i_1}\cdots v_{n,i_n}\,b_{i_1,...,i_n}\,,
\endaligned
\end{equation*}
where $v_{k,1},...,v_{k,r}\in{\mathbb R}$
are the coordinates of the vector $v_k\in{\mathbb R}^r$ represented
in the basis $\mathbf h_1,...,\mathbf h_r$ \ $(k=1,...,n)$.
\end{lemma}

The next lemma collects some simple and useful properties of $B$-valued $n$-morphisms on a  (possibly non-abelian) group $G$.
\begin{lemma}\label{Lm: spaces of polylinear forms are of finite dim}
$(a)$ ${\mathcal L}_n(G,B)\subseteq {\mathcal C}^n(G,B)$ for every $n\in\mathbb N$;
in other words, every $n$-morphism $L$ is also a $n$-cochain,
that is, $L(g_1,...,g_n)=0$ whenever one of the elements $g_1,...,g_n$
equals $\mathbf e$.
\smallskip

\noindent $(b)$ For any $L\in \mathcal L_n(G,B)$,
$L(g_1,...,g_n)=0$ whenever one of the elements $g_1,...,g_n$
belongs to the commutator subgroup $[G,G]$.
\smallskip

\noindent $(c)$ Consider any $L\in \mathcal L_n(G,B)$ and
$\widetilde {g_1},...,\widetilde {g_n}\in\widetilde G=G/[G,G]$.
Take arbitrary representatives  $g_i\in\widetilde{g_i}$, $1\le i\le n$,
and set
\begin{equation*}\label{factorization of n-homomorphism}
\widetilde L(\widetilde {g_1},...,\widetilde {g_n})
:=L(g_1,...,g_n)\,.
\end{equation*}
Then $\widetilde L$ is a well-defined $B$-valued $n$-morphism
of the abelianization $\widetilde G$ of $G$.
The correspondence $L\mapsto \widetilde L$ defines an isomorphism
of groups
\begin{equation*}\label{isomorphism mu}
\mu\colon {\mathcal L}_n(G,B)\ni L\mapsto\mu(L)
=\widetilde L\in {\mathcal L}_n(\widetilde G,B)\,.
\end{equation*}
Moreover, $\mu$ is an isomorphism of vector spaces whenever $B$ is a vector space.
\smallskip

\noindent $(d)$ Suppose that $B$ is a vector space over a field
$F$ of characteristic $0$,
$\dim_F B<\infty$, and the abelianization $\widetilde G$ of $G$
is finitely generated. Then
\footnote{Although this fact is simple and well-known, we treat it here just for completeness
of the proof of our main result. Notice that the assumption $\charac F=0$
could be omitted.}
\begin{equation*}\label{dimension of n-linear form space}
\dim_F {\mathcal L}_n(G,B)=\dim_F {\mathcal L}_n(\widetilde G,B)<\infty
\quad\text{for every} \ n\in\mathbb N\,.
\end{equation*}
\end{lemma}

\begin{proof}
$(a)$ It suffices to show that $L(\mathbf e, g_2,...,g_n)=0$. The latter
relation follows from the following identities:
$$
L(\mathbf e, g_2,...,g_n)=L(\mathbf e\cdot\mathbf e, g_2,...,g_n)
=L(\mathbf e, g_2,...,g_n)+L(\mathbf e, g_2,...,g_n)\,.
$$
$(b)$ Using (\ref{polymorphism}) and the fact that $B$ is abelian, we get that for any $h_1, h_2, g_2, \ldots, g_n \in G$, the following identities hold:
$$
L(h_1h_2h_1^{-1}h_2^{-1}, g_2,...,g_n)
=L(h_1h_1^{-1}h_2h_2^{-1}, g_2,...,g_n)=L(\mathbf e, g_2,...,g_n)=0\,.
$$
This implies $(b)$.
\smallskip

\noindent $(c)$ This is an immediate consequence of $(b)$.
\smallskip

\noindent $(d)$ Since $\widetilde G$ is a finitely generated abelian
group, we have $\widetilde G\cong \mathbb Z^r\oplus K$, where
$r\in\mathbb Z_+$ and $K$ is a finite abelian group.
For any ${\widetilde L}\in {\mathcal L}_n(\widetilde G,B)$
we have ${\widetilde L}(\widetilde {g_1},..., \widetilde {g_n})=0$
whenever at least one of the elements
$\widetilde {g_1},..., \widetilde {g_n}$
belongs to the subgroup $K$.
Indeed, assuming, for instance, that $\widetilde {g_1}$ is an element
of finite order $m \in \mathbb{N}$, we see that
$$
m\cdot{\widetilde L}(\widetilde {g_1},..., \widetilde {g_n})
={\widetilde L}(\widetilde {g_1}^m,..., \widetilde {g_n})=
{\widetilde L}(\widetilde {\mathbf e},..., \widetilde {g_n})=0\,,
$$
which implies ${\widetilde L}(\widetilde {g_1},..., \widetilde {g_n})=0$ (note that $\charac F=0$). It follows from $(c)$ and the above fact that
$$
\mathcal L_n(G,B)\cong\mathcal L_n(\widetilde G,B)\cong \mathcal L_n(\mathbb Z^r,B)\,.
$$
Therefore, it suffices to show that $\dim_F \mathcal L_n(\mathbb Z^r,B)<\infty$.
\smallskip

\noindent Let $\mathbf h_1=(1,0,...,0),...,\mathbf h_r=(0,...,0,1)$
be the standard basis of $\mathbb Z^r$ and
$\mathbf b^1,...,\mathbf b^s$ be a basis in \nolinebreak $B$.
\smallskip

\noindent Let $\mathcal I$ be the family of all finite ordered sets
$I=(i_1,...,i_n)$ consisting of $n$ natural numbers
$i_1,...,i_n$ such that $1\le i_1,...,i_n\le r$
(clearly, $\#\mathcal I = r^n$).
For each $I=(i_1,...,i_n)\in\mathcal I$ and each $i=1,...,s$,
we define the $n$-linear form $\ell_I^i\in{\mathcal L}^n(\mathbb Z^r,B)$
as follows:
\begin{equation}\label{basis in space of polylinear forms}
\ell_I^i(\mathbf h_{j_1},...,\mathbf h_{j_n})
    =\left\{
        \aligned
            \ &\mathbf b^i \ \ \text{if} \ \, (j_1,...,j_n)=I\,,\\
              & 0\hskip0.44cm\text{otherwise}\,;
        \endaligned
      \right.
\end{equation}
These $sr^n$ forms $\ell_I^i$ constitute a basis of
${\mathcal L}_n(\mathbb Z^r,B)$.\footnote{Obviously, these forms $\ell_I^i$ are linearly independent.}
Indeed, for any $L\in{\mathcal L}_n(\mathbb Z^r,B)$, it suffices to show that $L$ is a linear combination of $\ell_I^i$. Let us consider $I=(i_1,...,i_n)\in\mathcal I$.
Since $L(\mathbf h_{i_1},...,\mathbf h_{i_n})\in B$
and $\{\mathbf b^1,...,\mathbf b^s\}$ is a basis of $B$, we have
\begin{equation}\label{LI(hi1,...,hin) via basis b1,...,bs}
L(\mathbf h_{i_1},...,\mathbf h_{i_n})
=\sum_{i=1}^s a_I^i\,\mathbf b^i
\quad\text{for some} \ \, a_I^1,...,a_I^s\in F\,.
\end{equation}
It follows from (\ref{basis in space of polylinear forms}) and (\ref{LI(hi1,...,hin) via basis b1,...,bs}) that
$$ L(\mathbf h_{i_1},...,\mathbf h_{i_n})
=\sum_{i=1}^s a_I^i\,\mathbf \ell_I^i(\mathbf h_{i_1},...,\mathbf h_{i_n})=\sum_{i=1}^s\sum_{J\in\mathcal I}  a_J^i\,\ell_J^i(\mathbf h_{i_1},...,\mathbf h_{i_n}).$$
This proves that
$$
L=\sum_{I\in\mathcal I} \sum_{i=1}^s a_I^i\,\ell_I^i\,,
$$
which concludes the proof.
\end{proof}
\subsection{Embeddings
$\left[\mathcal P^n(G,A)/\mathcal P^{n-1}(G,A)\right]\hookrightarrow\mathcal C^n(G,A^G)$}
\label{subsect: Embeddings of polynomials to cochains with values in AG}
By Definition \ref{Def: G-polynomials in A}, the iterated
difference operator $D^{n+1}$ annihilates the subgroup
${\mathcal P}_n={\mathcal P}_n(G,A)=\ker D^{n+1}\subseteq A$ of all polynomial-like
elements of order at most $n$. However, the preceding iterated
difference operator $D^n$ may be non-zero on ${\mathcal P}_n$.
In this section, we study certain properties of the restriction
\begin{equation*}\label{restriction of Dn to Pn}
D^n_{\mathcal P_n}\Def \left. D^n\,\right|_{\mathcal P_n}
\colon\mathcal P_n\to\mathcal C^n(G,A)
\end{equation*}
of the operator $D^n\colon A\to\mathcal C^n(G,A)$
to the subgroup ${\mathcal P}_n\subseteq A$.
\smallskip

\noindent According to our notation, ${\mathcal C}^n(G,A^G)$ is
the group of all $n$-cochains of $G$ with values in
the subgroup $A^G\subseteq A$ consisting of all $G$-invariant
elements in $A$. Clearly, ${\mathcal C}^n(G,A^G)$ is
a subgroup of the group $\mathcal C^n(G,A)$
(if $A$ is a vector space then ${\mathcal C}^n(G,A^G)$ is
a vector subspace of the vector space $\mathcal C^n(G,A)$).

\begin{lemma}
\label{Lm: mapping of polynomials to invariant cochains}
The image of $D^n_{{\mathcal P}_n}$ is contained in the subgroup
${\mathcal C}^n(G,A^G)\subseteq {\mathcal C}^n(G,A)$, i.e., \
$$
D^n({\mathcal P_n})\subseteq \mathcal C^n(G,A^G)\,.
$$
That is, $D^n_{{\mathcal P}_n}$ may be regarded as a group homomorphism
$($a linear operator in case of vector spaces$)$
\begin{equation*}\label{operator DnP}
D^n_{{\mathcal P}_n}\colon {\mathcal P}_n(G,A)\to{\mathcal C}^n(G,A^G)
\end{equation*}
acting from the group ${\mathcal P}_n(G,A)$ to the group ${\mathcal C}^n(G,A^G)$.
\smallskip

\noindent By {\rm (\ref{Pn is contained in P{n+1}}),
(\ref{restriction of Dn to Pn})}, and
{\rm Definition \ref{Def: G-polynomials in A}},
the kernel $\ker D^n_{{\mathcal P}_n}$ of $D^n_{{\mathcal P}_n}$
coincides with $\mathcal P_{n-1}(G,A)$. Thus, the homomorphism
$D^n_{{\mathcal P}_n}$  gives rise to the exact sequence
\begin{equation}\label{1st exact sequence}
\CD
0@>>>{\mathcal P}_{n-1}(G,A)@>>>{\mathcal P}_n(G,A)@>{D^n_{{\mathcal P}_n}}>>
{\mathcal C}^n(G,A^G),\\
\endCD
\end{equation}
and the embedding of the quotient group
\begin{equation}\label{embedding in}
i_n\colon [{\mathcal P}_n(G,A)/{\mathcal P}_{n-1}(G,A)]
\hookrightarrow {\mathcal C}^n(G,A^G)\,.
\end{equation}
\end{lemma}

\begin{proof}
Let $n\in\mathbb Z_+$. Take any $a\in {\mathcal P}_n(G,A)$.
By Definition \ref{Def: G-polynomials in A}, $D^{n+1}a=0$.
According to (\ref{definition of iterated difference operators}),
we have
\begin{equation*}\label{equation D(n+1)a=0}
0=D^{n+1}a =d_{n+1}(D^n a)\,.
\end{equation*}
In other words, we get
$$
\aligned
0&=[d_{n+1}(D^n a)](g_1,...,g_{n+1})
=[(D^n a)^{g_{n+1}}-D^n a](g_1,...,g_n)\\
&=\{[D^n a](g_1,...,g_n)\}^{g_{n+1}}-[D^n a](g_1,...,g_n) \ \
\text{for all} \ \, g_1,...,g_{n+1}\in G\,.
\endaligned
$$
The latter relation may be written as
\begin{equation*}\label{equation [D(n+1)a](g1,...,gn+1)=0}
\{[D^n a](g_1,...,g_n)\}^g=[D^n a](g_1,...,g_n)
\ \ \text{for all} \ \, g,\,g_1,...,g_n\in G\,,
\end{equation*}
which actually means that all the values $[D^n a](g_1,...,g_n)$
of the $n$-cochain $D^n a$ are $G$-invariant. 
\medskip

\noindent Hence, for every $a\in{\mathcal P}_n$ we may regard
the $n$-cochain $D^n_{{\mathcal P}_n}(a)=D^n a$
(which is originally an $A$-valued one) as
a $n$-cochain of $G$ {\em with values in} $A^G$: \
$D^n_{{\mathcal P}_n}(a)=D^n a\in{\mathcal C}^n(G,A^G)$.
So $D^n_{{\mathcal P}_n}({\mathcal P}_n)\subseteq{\mathcal C}^n(G,A^G)$
and $D^n_{{\mathcal P}_n}$ may be regarded as a homomorphism (linear operator)
acting from ${\mathcal P}_n(G,A)$ to ${\mathcal C}^n(G,A^G)$.
\medskip

\noindent The rest statements of Lemma \ref{Lm: mapping of polynomials to invariant cochains} follow immediately from
what we have proved above.
\end{proof}
\medskip
\subsection{Coboundary operators and iterated difference operators}
\label{subsect: Coboundary operators and iterated difference operators}
\noindent In this subsection, we want to study more carefully the image
$D^n_{{\mathcal P}_n}({\mathcal P}_n)\subseteq{\mathcal C}^n(G,A^G)$ of the homomorphism
$D^n_{{\mathcal P}_n}$. We will see that this image is contained
in the group ${\mathcal L}_n(G,A^G)$ of all $A^G$-valued $n$-morphism
 $G$ (according to Lemma \ref{Lm: spaces of polylinear
forms are of finite dim}$(a)$,
${\mathcal L}_n(G,A^G)\subseteq{\mathcal C}^n(G,A^G)$). To prove this,
we will
introduce the {\em coboundary operators}
$\delta^n\colon {\mathcal C}^n(G,A)\to{\mathcal C}^{n+1}(G,A)$
(related to the {\em trivial} left $G$-action in $A$)
and study their relation to the iterated difference operators $D^n$.

\begin{definition}
Let us consider $A$ as a trivial left $G$-module and
define the {\em coboundary operator}
$$
\delta^n\colon\, {\mathcal C}^n(G,A)\to{\mathcal C}^{n+1}(G,A)
$$
by the standard formula
\begin{equation*}\label{coboundary operator}
\aligned
\left[\delta^n c\right](g_1,\ldots,&g_{n+1}) = c(g_2,\ldots,g_{n+1})
-c(g_{1}g_2,g_3,\ldots,g_{n+1})
+ c(g_1,g_{2}g_3,\ldots,g_{n+1})\\
&+\ldots + (-1)^n c(g_1, \ldots ,g_{n-1},g_n g_{n+1})
+(-1)^{n+1}c(g_1, \ldots ,g_n)\,.
\endaligned
\end{equation*}
Notice that $\delta^0=0$ since we use the trivial left $G$-action in $A$;
indeed, $\mathcal C^0(G,A)=A$ and for any $a\in A$ and $g\in G$
we have $(\delta^0 a)(g)={\,}^g a-a=a-a=0$.
\end{definition}

%
%

\begin{lemma}\label{Lm: polymorphism is a cocycle}
Any $n$-morphism is a $n$-cocycle.
\end{lemma}

\pf
According to Lemma \ref{Lm: spaces of polylinear
forms are of finite dim}$(a)$, any
$n$-morphism $c$ is a $n$-cochain. A straightforward calculation
shows that $\delta^n c(g_1,...,g_{n+1})=0$ for all $g_1,...,g_{n+1}$. \ep

In the following two lemmas, we establish some simple algebraic properties
of the operators $\delta$, $d$, and $D$. We omit a simple computation that proves the first lemma.

\begin{lemma}\label{Lm: relation between delta{n} and delta{n-1}}
Let $n\in\mathbb N$ and $c\in {\mathcal C^n}(G,A)$.
For each $h\in G$, we define the $(n-1)$-cochain $c_h$ by
$$
c_h(h_1,...,h_{n-1}):=c(h,h_1,...,h_{n-1})\,.
$$
Then for every $h,g_1,...,g_n\in G$, we have
$$
\aligned
\left[\delta^n c\right] (h,g_1,...,g_n)
&=c(g_1,g_2,\ldots,g_n)-c(hg_1,g_2,\ldots,g_n)\\
&\hskip90pt+ c(h,g_{2},\ldots,g_n)-[\delta^{n-1} c_h](g_1,\ldots,g_n)\,.
\endaligned
$$
\end{lemma}


\begin{lemma}\label{Lm: relation between d and delta}
For any $n\ge 1$ we have:
\begin{eqnarray}
\delta^n d_{n-1} &=& d_n \delta^{n-1} + (-1)^n d_n d_{n-1}\,,
\label{relation between d, D, and delta}\\
\delta^n D^n \ \ &=& d_n \delta^{n-1} D^{n-1} + (-1)^n D^{n+1}\,,
\label{relation between D and delta}\\
\delta^n D^n  \ \
&=& \left\{
   \aligned
\ & -D^{n+1} \ \ \text{for} \ n \ \text{odd}\,,\\
  &\ \ \ \ \, 0   \ \qquad \text{for} \ n \ \text{even}\,.
   \endaligned
  \right .\qquad \label{relation between d and D}
\end{eqnarray}
\end{lemma}

\pf
Let $c\in \mathcal C^{n-1}(G,A)$ and
$g_1,...,g_{n+1}\in G$. Put $c':= d_{n-1}c$.
Then we have:
$$
\aligned
(&\delta^n d_{n-1}c)(g_1,...,g_{n+1}) = (\delta^nc')(g_1,...,g_{n+1})\\
&= c'(g_2,...,g_{n+1}) - c'(g_1g_2,g_3,...,g_{n+1})\\
&\hskip10pt+ c'(g_1,g_2g_3,g_4,...,g_{n+1})
+... +(-1)^{n-1}c'(g_1,...,g_{n-1}g_n,g_{n+1})\\
&\hskip10pt+ (-1)^nc'(g_1,...,g_{n-1},g_ng_{n+1})
+(-1)^{n+1}c'(g_1,...,g_{n-1},g_n)\\
&= (d_{n-1}c)(g_2,...,g_{n+1}) - (d_{n-1}c)(g_1g_2,g_3,...,g_{n+1})\\
&\hskip10pt+ (d_{n-1}c)(g_1,g_2g_3,g_4,...,g_{n+1})
+... +(-1)^{n-1}(d_{n-1}c)(g_1,...,g_{n-1}g_n,g_{n+1})\\
&\hskip10pt+ (-1)^n(d_{n-1}c)(g_1,...,g_{n-1},g_ng_{n+1})
+ (-1)^{n+1}(d_{n-1}c)(g_1,...,g_{n-1},g_n)\\
&= [c^{g_{n+1}}(g_2,...,g_n) - c(g_2,...,g_n)]
- [c^{g_{n+1}}(g_1g_2,g_3,...,g_n) - c(g_1g_2,g_3,...,g_n)]\\
&\hskip10pt + [c^{g_{n+1}}(g_1,g_2g_3,g_4,...,g_n)
                               - c(g_1,g_2g_3,g_4,...,g_n)]\\
&\hskip10pt+... + (-1)^{n-1}[c^{g_{n+1}}(g_1,...,g_{n-2},g_{n-1}g_n)
- c(g_1,...,g_{n-2},g_{n-1}g_n)]\\
&\hskip10pt +(-1)^n [c^{g_ng_{n+1}}(g_1,...,g_{n-1})- c(g_1,...,g_{n-1})]\\
&\hskip10pt+ (-1)^{n+1}[c^{g_n}(g_1,...,g_{n-1}) - c(g_1,...,g_{n-1})]
\endaligned
$$
Adding the zero sum
$
(-1)^n c^{g_{n+1}}(g_1,...,g_{n-1})-(-1)^n c^{g_{n+1}}(g_1,...,g_{n-1})
$
to the right hand side of the above equality and rearranging the obtained expression to see that
$$
\aligned
&(\delta^n d_{n-1}c)(g_1,...,g_{n+1})\\
&=[c(g_2,...,g_n) - c(g_1g_2,...,g_n) + c(g_1,g_2g_3,...,g_n)\\
&\hskip50pt  +... + (-1)^{n-1}c(g_1,...,g_{n-1}g_n)
                       + (-1)^nc(g_1,...,g_{n-1})]^{g_{n+1}}\\
&\hskip50pt - [c(g_2,...,g_n) - c(g_1g_2,...,g_n) + c(g_1,g_2g_3,...,g_n)\\
&\hskip50pt +... + (-1)^{n-1}c(g_1,...,g_{n-1}g_n)+(-1)^nc(g_1,...,g_{n-1})]\\
&\hskip50pt +(-1)^n[c^{g_ng_{n+1}}(g_1,...,g_{n-1})
- c^{g_{n+1}}(g_1,...,g_{n-1})\\
&\hskip50pt - c^{g_n}(g_1,...,g_{n-1})+ c(g_1,...,g_{n-1})]\\
&= [(\delta^{n-1}c)(g_1,...,g_n)]^{g_{n+1}} - (\delta^{n-1}c)(g_1,...,g_n)\\
&\hskip50pt +(-1)^n \{[(d_{n-1}c)(g_1,...,g_n)]^{g_{n+1}} - (d_{n-1}c)(g_1,...,g_n)\} \\
&= [d_n(\delta^{n-1}c)](g_1,...,g_{n+1})
+ (-1)^n [d_n (d_{n-1}c)](g_1,...,g_{n+1})\\
&= [(d_n\delta^{n-1} + (-1)^n d_nd_{n-1})c](g_1,...,g_{n+1})\,,
\endaligned
$$
which proves (\ref{relation between d, D, and delta}).
Furthermore, by (\ref{definition of iterated difference operators})
and (\ref{relation between d, D, and delta}), we have
$$
\aligned
\delta^n D^n &= \delta^n d_{n-1} d_{n-2}\cdots d_0
= (\delta^n d_{n-1})\, d_{n-2}\cdots d_0\\
&= (d_n\delta^{n-1} + (-1)^n d_n d_{n-1})\, d_{n-2}\cdots d_0\\
&= d_n\delta^{n-1} d_{n-2}\cdots d_0 + (-1)^n d_n d_{n-1}d_{n-2}\cdots d_0\\
&= d_n\delta^{n-1}D^{n-1} + (-1)^n D^{n+1}\,;
\endaligned
$$
this proves (\ref{relation between D and delta}).
Finally, to prove (\ref{relation between d and D}),
we use (\ref{relation between D and delta}) and induction in $n$.
Since $\delta^0 = 0$, we have $\delta^0 D^0 = 0$.
Assume that for some $m\in \mathbb N$, relation (\ref{relation between d and D})
is already proven for all $n\le m-1$. If $m$ is odd
then $m-1$ is even and
$$
\delta^mD^m = d_m\delta^{m-1}D^{m-1} + (-1)^m D^{m+1} = - D^{m+1}\,.
$$
If $m$ is even then $m-1$ is odd and
$$
\aligned
\delta^mD^m &= d_m\delta^{m-1}D^{m-1} + (-1)^m D^{m+1}\\
&= d_{m} (-D^m) + D^{m+1} = -D^{m+1} + D^{m+1} = 0.
\endaligned
$$
This competes the proof of the lemma.
\ep

\begin{corollary}\label{Crl: Dna is n-cocycle whenever D{n+1}a=0}
Let $n\in\mathbb Z_+$ and $a\in{\mathcal P}_n(G,A)$.
Then $D^n a$ is a $n$-cocycle of $G$ with values in $A^G$.
\end{corollary}

\pf
By Lemma \ref{Lm: mapping of polynomials to invariant cochains},
$D^n a\in{\mathcal C}^n(G,A^G)$. Thus, we only need to check that
$\delta^n D^n a=0$. By (\ref{relation between d and D}),
$$
\delta^n D^n = \left\{
   \aligned
     \ & -D^{n+1} \ \ \text{for} \ n \ \text{odd}\,,\\
       &\ \ \ \ \, 0   \ \qquad \text{for} \ n \ \text{even}\,;
   \endaligned
  \right .
$$
this implies $\delta^n D^n a=0$ since
$D^{n+1}a=0$.
\ep
\subsection{Iterated difference operators and polymorphisms}
\label{subsect: Iterated difference operators and polymorphisms}
Here we establish that polynomial-like elements are closely related
to polymorphisms. This connection provides us
with an appropriate tool for the proving that
the spaces of polynomial-like elements are of finite dimension, i.e., the Main Theorem.
\smallskip

\begin{notation} Let ${\mathcal L}_n^{\mathcal S}(G,A^G)$ denote
the subgroup of the group ${\mathcal L}_n(G,A^G)$
consisting of all {\em symmetric} $A^G$-valued $n$-morphisms
of $G$. In more details, a $n$-morphism
$$L\colon\, \underset{n}{\underbrace{G\times\cdots\times G}}\to A^G$$
belongs to ${\mathcal L}_n^{\mathcal S}(G,A^G)$
whenewer $L(g_{j_1},...,g_{j_n})=L(g_1,...,g_n)$
for any $g_1,...,g_n\in G$ and any permutation $(j_1,...,j_n)$
of the indices $1,...,n$. \hfill $\bigcirc$
\end{notation}
\medskip

\noindent Suppose $a\in{\mathcal P}_0(G,A)$, that is, $a\in A$
and $[D^1 a](g)=a^g-a=0$ for all all $g\in G$. This shows that
the element $a$ is in $A^G$ and hence its image $D^0 a=a$ belongs to
$A^G={\mathcal L}_0(G,A^G)$
(see Definition \ref{Def: group polymorphism}).
\smallskip

\noindent Suppose now that $a\in{\mathcal P}_1(G,A)$.
Then, according to Corollary \ref{Crl: Dna is n-cocycle whenever D{n+1}a=0},
$D^1 a$ is a $1$-cocycle of $G$ with values in $A^G$. Hence
\begin{equation*}\label{D^1a is 1-cocycle}
\aligned
\ &D^1 a\colon G\to A^G \quad\text{and}\quad \delta^1(D^1 a)=0\,,
\quad\text{that is,}\\
&[D^1 a](g_2)-[D^1 a](g_1g_2)+[D^1 a](g_1)=0 \ \,
\text{for all} \ \, g_1,g_2\in G\,,
\endaligned
\end{equation*}
which shows that the mapping $D^1 a\colon G\to A^G$
{\sl is a group homomorphism whenever} $a\in{\mathcal P}_1(G,A)$.
\smallskip

\noindent Combining this simple observation
with Lemma \ref{Lm: D^n a is symmetric for a in Pn},
we come to the following property:

\begin{proposition}\label{Prp: DnPn(G,AG) is contained in LnS(G,AG)}
Let $n\in\mathbb Z_+$ and $a\in{\mathcal P}_n(G,A)$. Then
$D^n a\in{\mathcal L}_n^{\mathcal S}(G,A^G)$,
that is, $D^n a$ is a symmetric $A^G$-valued $n$-morphism of $G$.
\end{proposition}

\pf
The symmetry of $D^n a$ and the inclusion
$[D^n a](g_1,...,g_n)\in A^G$ for any $a\in{\mathcal P_n}(G,A)$
and all $g_1,...,g_n\in G$ have been established in
Lemmas \ref{Lm: D^n a is symmetric for a in Pn}
and \ref{Lm: mapping of polynomials to invariant cochains}.
Thus, we only need to prove that $D^n a$ is a polymorphism.
\smallskip

\noindent The proof is by induction in $n$. The cases $n=0,1$
have been already handled above. Suppose that $m\ge 2$ and
the statement is already proven for $n=m-1$. To complete the proof,
we need to show that $D^m a\in{\mathcal L}_m(G,A^G)$ whenever $D^{m+1}a=0$.
\smallskip

\noindent Define cochains $c\in{\mathcal C}^m(G,A^G)$ and
$c_h\in{\mathcal C}^{m-1}(G,A^G)$ ($h\in G$) by
\begin{equation}\label{c and ch}
\aligned
c(g_1,...,g_m) \,\ &=[D^m a](g_1,g_2,...,g_m)\quad\text{and}\\
c_h(g_2,...,g_m) &=c(h,g_2,...,g_m)=[D^m a](h,g_2,...,g_m)\,.
\endaligned
\end{equation}
We start with the following claim:
\medskip

\noindent{\smbfit Claim 1.} {\sl For any $h\in G$, the $(m-1)$-cochain
$c_h$,
$$
c_h\colon\,\underset{m-1}{\underbrace{G\times\cdots\times G}}
\ni (g_2,...,g_m)\mapsto c_h(g_2,...,g_m)=[D^m a](h,g_2,...,g_m)\in A^G\,,
$$
is an $A^G$-valued $(m-1)$-homomorphism.}
\medskip

\noindent To justify this claim, we apply
Lemma \ref{Lm: [D{n+1}a](g1,...,g{n+1})=
[Dnag1](g2,...,g{n+1}) - [Dna](g2,...,g{n+1})}
to the $(m+1)$-cochain $D^{m+1} a$. By taking into account that
$D^{m+1} a=0$, for any $h_1,...,h_m\in G$, we obtain
\begin{equation*}
\begin{split}
[D^m (a^h-a)](h_1,...,h_m)&=[D^m (a^h)](h_1,...,h_m) - [D^m a](h_1,...,h_m)\\
&=[D^{m+1} a](h,h_1,...,h_m)=0\,.
\end{split}
\end{equation*}
The latter relation means that the element $b_h\Def a^h-a$
is a polynomial-like element in $A$ of order at most $m-1$.
Hence, by the induction hypothesis,
$D^{m-1} b_h\in{\mathcal L}_{m-1}(G,A^G)$.
Using Lemma \ref{Lm: [D{n+1}a](g1,...,g{n+1})=
[Dnag1](g2,...,g{n+1}) - [Dna](g2,...,g{n+1})} again, we obtain
that for any $g_2,...,g_m\in G$
$$
\aligned
c_h(g_2,...,g_m)&=\lbrack D^m a\rbrack (h,g_2,...,g_m)
=[D^{m-1}(a^h)](g_2,...,g_m)-[D^{m-1} a](g_2,...,g_m)\\
&=[D^{m-1}(a^h-a)](g_2,...,g_m)
=[D^{m-1} b_h](g_2,...,g_m)\,.
\endaligned
$$
Thereby, $c_h=D^{m-1} b_h\in{\mathcal L}_{m-1}(G,A^G)$,
which proves Claim 1.

\noindent In view of Claim 1, the following statement is
an immediate consequence of Lemma \ref{Lm: polymorphism is a cocycle}:
\medskip

\noindent{\smbfit Claim 2.} {\sl For any $h\in G$,
the $(m-1)$-cochain $c_h\in{\mathcal L}_{m-1}(G,A^G)$ defined in
{\rm (\ref{c and ch})}
is a $(m-1)$-cocycle, that is, $[\delta^{m-1} c_h](g_1,...,g_m)=0$
for every $g_1,...,g_m\in G$.}
\medskip

\noindent{\smbfit Claim 3.} {\sl For any $g_2,...,g_m\in G$ the mapping
\begin{equation}\label{linearity in first argument}
F_{g_2,...,g_m}\colon\, G\ni g\mapsto
[D^m a](g,g_2,...,g_m)\in A^G
\end{equation}
is a group homomorphism.}
\medskip

\noindent To justify Claim 3, we apply
Lemma \ref{Lm: relation between delta{n} and delta{n-1}}
(with $n=m+1$) to the $m$-cochain $c=D^m a$ and $(m-1)$-cochain
$c_h$ defined in (\ref{c and ch}).
This gives the relation
\begin{equation}\label{main relation}
\aligned
\hskip-5pt\left[\delta^m D^m a\right]&(h,g_1,...,g_m)
=[\delta^m c] (h,g_1,...,g_m)\\
&=c(g_1,g_2,\ldots,g_m)-c(hg_1,g_2,\ldots,g_m)\\
&\hskip20pt+ c(h,g_2,\ldots,g_m)-[\delta^{m-1} c_h](g_1,\ldots,g_m)\\
&=[D^m a](g_1,g_2,\ldots,g_m)-[D^m a](hg_1,g_2,\ldots,g_m)\\
&\hskip20pt + [D^m a](h,g_2,\ldots,g_m)-[\delta^{m-1} c_h](g_1,\ldots,g_m)\,.
\endaligned
\end{equation}
Since $D^m a$ is a $m$-cocycle
(Corollary \ref{Crl: Dna is n-cocycle whenever D{n+1}a=0})
and $c_h$ is a $(m-1)$-cocycle (Claim 2),
relation (\ref{main relation}) takes the form
$$
[D^m a](hg_1,g_2,\ldots,g_m)=[D^m a](h,g_2,\ldots,g_m)
+[D^m a](g_1,g_2,\ldots,g_m)\,,
$$
which proves (\ref{linearity in first argument}).
\medskip

\noindent Combining Claim 1 with Claim 3, we complete the induction
step, which proves Proposition \ref{Prp: DnPn(G,AG) is contained in LnS(G,AG)}.
\ep

\begin{remark}\label{Rmk: more simple proof of Proposition
{Prp: DnPn(G,AG) is contained in LnS(G,AG)}}
In fact, the above proof of Proposition
\ref{Prp: DnPn(G,AG) is contained in LnS(G,AG)}
can be simplified:
whenever Claim 1 is proven, one may use the symmetry of
$D^n a$ (Lemma \ref{Lm: D^n a is symmetric for a in Pn}) to conclude
that for any $n\ge 2$, the mapping
$D^n a\colon\,\underset{n}{\underbrace{G\times\cdots\times G}}\to A^G$
is a group polymorphism. \hfill $\bigcirc$
\end{remark}

\begin{Main}\label{Thm: Main Theorem}
$(a)$ $D^n({\mathcal P}_n(G,A))\subseteq {\mathcal L}_n^{\mathcal S}(G,A^G)$
for every $n\in{\mathbb N}$. That is,
the image $D^n({\mathcal P}_n(G,A))$ of the subgroup
${\mathcal P}_n(G,A)\subseteq A$ of all polynomial-like elements of order
at most $n$ under the homomorphism
$D^n\colon {\mathcal P}_n(G,A)\to {\mathcal C}^n(G,A)$ is contained
in the subgroup ${\mathcal L}_n^{\mathcal S}(G,A^G)\subseteq {\mathcal C}^n(G,A)$
consisting of all symmetric $A^G$-valued $n$-morphisms of the group $G$.
\smallskip

\noindent $(b)$ Exact sequence {\rm (\ref{1st exact sequence})}
and embedding {\rm (\ref{embedding in})} may be written in the form
\begin{equation}\label{2nd exact sequence}
\CD
0@>>>{\mathcal P}_{n-1}(G,A)@>>>{\mathcal P}_n(G,A)@>{D^n_{{\mathcal P}_n}}>>
{\mathcal L}_n^{\mathcal S}(G,A^G)\\
\endCD
\end{equation}
and
\begin{equation*}\label{embedding in, 2nd form}
i_n\colon [{\mathcal P}_n(G,A)/{\mathcal P}_{n-1}(G,A)]
\hookrightarrow {\mathcal L}_n^{\mathcal S}(G,A^G)\,.
\end{equation*}
\smallskip

\noindent $(c)$ Suppose that $F$ is a field
of characteristic $0$,
$G$ is a group, and $A$ is a $F$-vector space endowed with
a linear right $G$-action.
If the group $\widetilde G=G/[G,G]$ is finitely generated
and $\dim_F A^G<\infty$ then
\begin{equation*}\label{again finiteness dimP}
\dim_F {\mathcal P}_n(G,A)<\infty\quad \text{for every} \ \, n\in\mathbb Z_+\,.
\end{equation*}
\end{Main}

\pf
$(a)$ is proven in Proposition \ref{Prp: DnPn(G,AG) is contained
in LnS(G,AG)}.
$(b)$ follows from $(a)$ and
Lemma \ref{Lm: mapping of polynomials to invariant cochains}.
To proof $(c)$, we use induction in $n$, statement $(b)$,
and Lemma \ref{Lm: spaces of polylinear forms are of finite dim}$(d)$.
Indeed, by assumption, the space ${\mathcal P}_0(G,A)=A^G$
is of finite dimension. Suppose that
we already have that $\dim_F {\mathcal P}_{n-1}(G,A)<\infty$.
By Lemma \ref{Lm: spaces of polylinear forms are of finite dim}$(d)$,
we also have that
$\dim_F {\mathcal L}_n^{\mathcal S}(G,A^G)
\le \dim_F {\mathcal L}_n(G,A^G)<\infty$.
Therefore, the induction hypothesis and the
exact sequence (\ref{2nd exact sequence}) would imply that
$\dim_F {\mathcal P}_n(G,A)<\infty$.
\ep

%
\noindent We immediately obtain the following estimate:
\begin{proposition}
\label{estimate-dim-Pn}
Under the same assumption of part ($c$) of Main Theorem, the following estimate holds:
$$\dim_F \mathcal{P}_n(G,A) \leq s \cdot \binom{n+r}{r},$$
where $s=\dim_F A^G$ and $r$ is the rank of the abelian group $\widetilde{G}$.
\end{proposition}

\pf
We prove it by induction in $n$.
The estimate is obviously true for $n=0$. So we assume that $n \geq 1$ and the estimate holds for $n-1$, i.e.,
\begin{equation}
\label{induction-estimate}
\dim_F \mathcal{P}_{n-1}(G,A) \leq s\cdot\binom{n+r-1}{r}.
\end{equation}
According to the proof of Main Theorem, we have
\begin{equation}
\label{estimate-mainthm}
\begin{split}
\dim_F \mathcal{P}_n(G,A)&=\dim_F \mathcal{P}_n(G,A)/ \mathcal{P}_{n-1}(G,A)+\dim_F \mathcal{P}_{n-1}(G,A)\\
&\leq \dim_F \mathcal{L}^{\mathcal{S}}_n(G,A^G)+\dim_F \mathcal{P}_{n-1}(G,A).
\end{split}
\end{equation}
Using Lemma \ref{Lm: spaces of polylinear forms are of finite dim}$(d)$,  we have
$$\dim_F \mathcal{L}^{\mathcal{S}}_n(G,A^G)=\dim_F \mathcal{L}^{\mathcal{S}}_n(\mathbb{Z}^r,A^G).$$
Moreover, we recall from the proof of Lemma \ref{Lm: spaces of polylinear forms are of finite dim}$(d)$ that
the polymorphisms $\{\ell^i_I\}_{1 \leq i \leq s, I \in \mathcal{I}}$ constitute a vector basis for $\mathcal{L}_n(\mathbb{Z}^r,A^G)$, where $\mathcal{I}$ is the family of all finite ordered sets $I$ consisting of $n$ natural numbers in $\{1, \ldots, r\}$.
For any $I=(i_1, \ldots, i_n) \in \mathcal{I}$, we define $\gamma(I):=(\gamma_1(I), \ldots, \gamma_r(I))$, where for each $1 \leq k \leq r$, the notation
$\gamma_k(I)$ is the number of indices $j \in \{1, \ldots, n\}$ such that $i_j=k$. If $I,J \in \mathcal{I}$ and $\gamma(I)=\gamma(J)$, we say that $I$ and $J$ belong to the same class.
Consider any $L \in \mathcal{L}^{\mathcal{S}}_n(\mathbb{Z}^r,A^G)$ and write it as follows
$$L=\sum_{i=1}^s\sum_{I\in\mathcal{I}}a^i_I \ell^i_I.$$
Due to the symmetry of $L$, whenever $I,J$ are in the same class, we must have $a^i_I=a^i_J$ for any $i \in \{1, \ldots, s\}$.
Thus, to determine $L$ completely, one just needs to know the coefficients $a^i_I$ for only one representative $I$ in each class. By definition, the cardinality of these classes is equal to the number of the $r$-tuples consisting of $r$ nonnegative intergers $(\gamma_1, \ldots, \gamma_r)$ satisfying $\gamma_1+\ldots \gamma_r=n$.
Therefore, we have
\begin{equation}
\label{dim-symmetric-morphism}
\dim_F \mathcal{L}^{\mathcal{S}}_n(G,A^G)=(\dim_F A^G)\cdot \binom{n+r-1}{r-1},
\end{equation}
for any $n \in \mathbb{Z}_+$. Using (\ref{induction-estimate}), (\ref{estimate-mainthm}), (\ref{dim-symmetric-morphism}),
we finish the proof.
\ep
\section{Polynomial-like elements in a ring}
\label{sect: Polynomial-like elements in a ring}

Suppose that a group $G$ acts from the right hand side in a ring $A$
in such a way that for any $g\in G$ the mapping $A\ni a\mapsto a^g\in A$
is a ring automorphism. Under this assumption, we prove now some additional properties of the iterated difference operators and the set $\mathcal{P}(G,A)$ of all polynomial-like elements in $A$.

\begin{proposition}\label{Prp: P(G,A) is a ring}
$(a)$ {\smbfit Leibniz-type formula:} For any $a,b\in A$ and $g\in G$,
\begin{equation}\label{Leibniz-type formula}
\ [D^1(ab)](g)=a^g\cdot \{[D^1 b](g)\} +\{[D^1 a](g)\}\cdot b\,.
\end{equation}
$(b)$ Let $m,n\in{\mathbb Z_+}$ and $m+n\ge 1$. Then
\begin{equation}\label{D{m+n-1}(ab)=0 whenever Dma=Dnb=0}
D^{m+n-1}(ab)=0\quad\text{whenever} \ \ a,b\in A \ \
\text{and} \ \, D^m a=D^n b=0\,,
\end{equation}
Equivalently, $ab\in{\mathcal P}_{m+n-1}(G,A)$ whenever
$a\in{\mathcal P}_m(G,A)$ and $b\in{\mathcal P}_n(G,A)$.
In particular, the set ${\mathcal P}(G,A)$ of all polynomial-like elements
in $A$ is a subring of the ring $A$.
\vskip0.3cm

\noindent $(c)$ Each operator $D^n\colon A\to {\mathcal C}^n(G,A)$, $n\in\mathbb N$,
is $A^G$-linear, i.e., whenever $a\in A$ is a $G$-invariant element, we have that $D^n(ab)=aD^n b$ for all $b\in A$.
\end{proposition}

\pf
\noindent $(a)$ (\ref{Leibniz-type formula}) is straightforward:
$$
\aligned
\left[D^1(ab)\right](g)&=(ab)^g-ab=a^g b^g-ab=a^g\cdot(b^g-b)+(a^g-a)\cdot b\\
&=a^g\cdot \{[D^1 b](g)\} +\{[D^1 a](g)\}\cdot b\,.
\endaligned
$$
$(b)$ The proof is by induction in $N=m+n$. Since $D^0=\id_A$,
statement (\ref{D{m+n-1}(ab)=0 whenever Dma=Dnb=0}) is trivial whenever one of the numbers $m,n$ is $0$;
hence (\ref{D{m+n-1}(ab)=0 whenever Dma=Dnb=0}) is true for $N=1$ and, from now on, we may assume
$m,n\in\mathbb N$. Suppose that for some $N\ge 1$, statement
(\ref{D{m+n-1}(ab)=0 whenever Dma=Dnb=0}) is true for all $m,n\in\mathbb N$ with $m+n=N$.
Take any $m,n\in{\mathbb N}$ with $m+n=N+1$ and assume $D^m a= D^n b=0$.
By Lemma \ref{Lm: [D{n+1}a](g1,...,g{n+1})=
[Dnag1](g2,...,g{n+1}) - [Dna](g2,...,g{n+1})},
$$
\aligned
\lbrack D^{m-1}(a^{g_1}-a)\rbrack (g_2,...,g_m)
&=[D^{m-1} (a^{g_1})](g_2,...,g_m) - [D^{m-1} a](g_2,...,g_m)\\
&=[D^m a](g_1,...,g_m)=0\,,
\endaligned
$$
that is,
\begin{equation}\label{f1}
D^{m-1} (a^{g_1}-a)=0 \ \ \text{for any} \ \, g_1\in G\,.
\end{equation}
Hence, taking into account the assumption $D^n b=0$ and
the induction hypothesis, we obtain
\begin{equation}\label{f2}
D^{m+n-2}[(a^{g_1}-a)b]=0\,.
\end{equation}
Similarly, the assumption $D^n b=0$ implies
\begin{equation}\label{f3}
D^{n-1} (b^{g_1}-b)=0 \ \ \text{for any} \ \, g_1\in G\,.
\end{equation}
combining this with the assumption $D^m a=0$ and
the induction hypothesis, we get
\begin{equation}\label{f4}
D^{m+n-2}[a(b^{g_1}-b)]=0\,.
\end{equation}
Summing up relations (\ref{f2}) and (\ref{f4}) we see that
$D^{m+n-2}[(a^{g_1}-a)b+a(b^{g_1}-b)]=0$, that is,
\begin{equation}\label{f5}
D^{m+n-2}(a^{g_1}b+ab^{g_1}-2ab)=0\,.
\end{equation}
Relation (\ref {f1}) certainly implies that
$D^m (a^{g_1}-a)=0$; in view of (\ref {f3}) and the
induction hypothesis, this shows that
$$
D^{m+n-2}[(a^{g_1}-a)(b^{g_1}-b)]=0\,.
$$
Thus,
\begin{equation}\label{f6}
D^{m+n-2}(a^{g_1}b^{g_1}-ab^{g_1}-a^{g_1}b+ab)=0\,.
\end{equation}
Summing up (\ref{f5}) and (\ref{f6}), we get
\begin{equation}\label{f7}
D^{m+n-2}(a^{g_1}b^{g_1}-ab)=0\,.
\end{equation}
Finally, we apply Lemma \ref{Lm: [D{n+1}a](g1,...,g{n+1})=
[Dnag1](g2,...,g{n+1}) - [Dna](g2,...,g{n+1})} to the element
$D^{m+n-1}(ab)$ and take into account (\ref{f7}) to see that for
every $g_1,g_2,...,g_{m+n-1}\in G$, we have
$$
\aligned
\lbrack D^{m+n-1}(ab&)\rbrack  (g_1,g_2,...,g_{m+n-1})\\
&=\{D^{m+n-2}[(ab)^{g_1}]\}(g_2,...,g_{m+n-1})
-[D^{m+n-2}(ab)](g_2,...,g_{m+n-1})\\
&=\lbrack D^{m+n-2}(a^{g_1}b^{g_1}-ab)\rbrack (g_2,...,g_{m+n-1})=0\,.
\endaligned
$$
Therefore, $D^{m+n-1}(ab)=0$, which completes the induction step
and proves $(b)$.

\noindent $(c)$
This follows immediately from
Lemma \ref{Lm: formula for Dna} and the fact that for any $g_1, \ldots, g_n \in G$ and $1 \leq i_1<\ldots <i_s \leq n$, one has
$$(ab)^{\pi_{i_1, \ldots, i_s}(g_1, \ldots, g_n)}=ab^{\pi_{i_1, \ldots, i_s}(g_1, \ldots, g_n)}.$$
\ep

\section{Polynomial-like elements related to lattices}
\label{sect: Polynomial-like elements related to lattice}


In this section, we now present an application of the Main Theorem to a concrete but important case: $G$ is a lattice which acts naturally \footnote{I.e., the action is isometric, free, properly discontinuous, and co-compact.} on an Euclidean space whose dimension is the same as the rank of the abelian group $G$. In this case, $A$ is a $G$-invariant subalgebra of the algebra of continuous functions defined on the Euclidean space (see Example \ref{Expl: polynomial-like elements in function spaces}).
We will first prove that polynomial-like elements of order $n$ are exactly polynomials of the same order with $G$-periodic coefficients. As a consequence, we obtain Corollary \ref{polynomial-like-solutions} which states roughly that given a periodic elliptic operator $\mathcal{D}$, all of the dimensions of the spaces of polynomial-like solutions of $\mathcal{D}$ of order $n$ ($n \in \mathbb{Z}_+$) are finite provided that there are only finitely many linealy independent periodic solutions. Notice that the same results should hold for the case of periodic elliptic operators on co-compact abelian coverings; however, for simplicity, here we focus first on the Euclidean case and then consider briefly the covering case.

\medskip
\subsection{Polynomials with periodic coefficients}
\label{subsec: G-periodic polynomials}
Let $G\cong{\mathbb Z}^r$ be a lattice of rank $r$ in ${\mathbb R}^r$.
This lattice acts naturally in the space $C(\mathbb R^r)$ of all
(real or complex) continuous functions on $\mathbb R^r$; namely \footnote{If we wish, this action may be regarded as a right one;
this is of no importance since $G$ is commutative.},
\begin{equation}\label{Gamma-action on functions}
\aligned
C(\mathbb R^r)\ni &f\mapsto f^g\to C(\mathbb R^r)\,,\\
\ \, &\text{where}\ \,
f^g(x)=f(T_g x)=f(x+g)\,, \ \ x\in\mathbb R^r\,, \ \,
g\in G\,.
\endaligned
\end{equation}
Actually, all mappings (\ref{Gamma-action on functions})
are automorphisms of the algebra $C(\mathbb R^r)$.
Denote by $C^G$ the subalgebra of $C(\mathbb R^r)$
consisting of all $G$-invariant (or, which is the same, $G$-periodic)
functions $f\in C(\mathbb R^r)$. For $n\in{\mathbb Z}_+$,
set
\begin{equation*}\label{Gamma-periodic polynomials}
\aligned
&P^G=C^G[x_1,...,x_r]\,,\\
&P^G_n=C^G_n[x_1,...,x_r]=\{p\in C^G[x_1,...,x_r]\mid \deg p\le n\}\,;
\endaligned
\end{equation*}
here $x_1,...,x_r$ are the standard coordinates in $\mathbb R^r$.
In other words, $P^G$ is the algebra (over $\mathbb C$ or $\mathbb R$)
of all polynomials in the coordinates $x_1,...,x_r$ with continuous
$G$-periodic coefficients; $P^G_n$ is the vector subspace of $P^G_n$
consisting of all polynomials in $x_1,...,x_r$ of degree at most $n$
with continuous $G$-periodic coefficients. By misuse of language, we call
elements in $P^G$ as {\em $G$-periodic polynomials} (or more specifically, {\em $G$-periodic polynomials of degree at most}
$n$ for elements in $P^G_n$). A $G$-periodic polynomial of the form
$Q_n=f(x)\cdot x_{i_1}\cdots x_{i_n}$ ($1\le i_1,...,i_n\le r$)
with $G$-periodic coefficient $f\ne 0$ is said to be
{\em $G$-periodic monomial of degree $n$};
such a monomial may also be written in the form
$Q_n=f(x)\cdot x_1^{j_1}\cdots x_r^{j_r}$, where $j_1,...,j_r\in{\mathbb Z}_+$
and $j_1+...+j_r=n$.
Any $G$-periodic polynomial of degree at most $n$
is a sum of $G$-periodic monomials of degree at most $n$;
such a representation is unique (up to an order of summands).
Clearly, all $P^G_n$ are $G$-invariant subspaces
of $C(\mathbb R^r)$, and $P^G$ itself is a $G$-invariant subalgebra
of $C(\mathbb R^r)$.
\medskip

\noindent The $G$-action in $C(\mathbb R^r)$ gives rise
to iterated difference operators
\begin{equation*}\label{Dn in C(Rr)}
D^n\colon C({\mathbb R}^r)\to {\mathcal C}^n(G,C({\mathbb R}^r))
\end{equation*}
(see (\ref{d homomorphisms}) and
(\ref{definition of iterated difference operators})).
As in Section \ref{sect: Polynomial-like elements in vector spaces with
group actions}, we define the subspaces
${\mathcal P}_n(G,C({\mathbb R}^r))\subseteq C({\mathbb R}^r)$
of {\em polynomial-like elements in $C({\mathbb R}^r)$ of order at most $n$}
by
\begin{equation*}\label{polynomial-like elements in C(Rr)}
{\mathcal P}_n(G,C({\mathbb R}^r))=\ker D^{n+1}\,
\end{equation*}
(see Definition \ref{Def: G-polynomials in A}).

\begin{lemma}\label{Lm: Dn of n-monomial}
$(a)$ $D^{n+1} Q_m=0$ for any $G$-periodic monomial of degree $m\le n$.
\smallskip

\noindent $(b)$ Let
$Q_n=x_{i_1}\cdots x_{i_n}=x_1^{j_1}\cdots x_r^{j_r}$, where
$1\le i_1,...,i_n\le r$, \ $j_1,...,j_r\in{\mathbb Z}_+$, \ $j_1+...+j_r=n$,
and each coordinate
$x_k$ $(k=1,...,r)$ enters the product $x_{i_1}\cdots x_{i_n}$
precisely $j_k$ times.
Let $g_1,...,g_n\in G$ and
$$
g_i=(g_{i,1},...,g_{i,r})=\sum_{k=1}^r g_{i,k}{\mathbf e}_k\quad
(1\le i\le n)\,,
$$
where $\{{\mathbf e}_1,...,{\mathbf e}_r\}$ is the standard basis
in ${\mathbb R}^r$. Then
\begin{equation}\label{Dn of n-monomial}
\aligned
\lbrack D^n Q_n \rbrack (g_1,...,g_n)
&=\sum_{{\boldsymbol\sigma}\in{\mathbf S}(n)}
g_{1,i_{s_1}}g_{2,i_{s_2}}\cdots g_{n,i_{s_n}}\\
&=j_1!\cdots j_r!\cdot\sum_{\boldsymbol\kappa}
g_{1,k_1}g_{2,k_2}\cdots g_{n,k_n}\,;
\endaligned
\end{equation}
here ${\boldsymbol\sigma}=(s_1,...,s_n)\in{\mathbf S}(n)$
in the first sum runs over all the permutations of the indices $1,...,n$,
whereas ${\boldsymbol\kappa}=(k_1,...,k_n)$ in the second sum runs over all
finite sequences of length $n$ containing each of the indices
$k\in\{1,...,r\}$ precisely $j_k$ times.
\end{lemma}

\pf
$(a)$ The proof is by induction in $n$.
If $f$ is a $G$-periodic function
then $f^g=f$ for all $g\in G$, that is,
\begin{equation}\label{D1(Gamma-periodic function)=0}
D^1 f=0\quad\text{for any} \ \, G\text{-periodic} \ \, f\,.
\end{equation}
For each coordinate function $\phi_i(x)=x_i$
($1\le i\le r$) and any $g=(g_1,...,g_r)\in G$,
we have
$$
\phi_i^g(x)=\phi_i(x+g)=x_i+g_i\quad
\text{and}\quad
\{[D^1\phi_i](g)\}(x)=\phi_i(x+g)-\phi_i(x)=g_i.
$$
Since the right hand side $g_i$ of the latter relation
does not depend on $x$, it is $G$-invariant and, therefore,
$$
\{[D^2\phi_i](g',g'')\}
=\{[D^1\phi_i](g')\}^{g''}-\{[D^1\phi_i](g')\}=0\,,
$$
which means that $D^2\phi_i=0$. Due to Proposition \ref{Prp: P(G,A) is a ring},
this implies that
\begin{equation}\label{D2(axi)=0}
D^2(a\phi_i)=0\quad\text{for all} \ \, i=1,...,r \
{\mathrm {whenever}} \ a\in C(\mathbb{R}^r)^G\,.
\end{equation}
Relations (\ref{D1(Gamma-periodic function)=0}), (\ref{D2(axi)=0})
provide us with a base of induction.
\smallskip

\noindent Suppose that statement $(a)$ is already proven
for some $n\ge 1$. Let $Q_m=f(x)\cdot x_{i_1}\cdots x_{i_m}$
be a $G$-periodic monomial of degree $m\le n+1$. If $m\le n$ then,
by the induction hypothesis, $D^{n+1} Q_m=0$, which
certainly implies $D^{n+2} Q_m=0$. On the other hand, if $m=n+1$,
we may represent the monomial $Q_{n+1}=Q_m$ as a product
$Q_{n+1}=Q_n\cdot x_{i_{n+1}}$, where $Q_n$ is a
$G$-periodic monomial of degree $n$.
By the induction hypothesis and (\ref{D2(axi)=0}), we have
$$
D^{n+1} Q_n=0\quad \text{and}\quad D^2 x_{i_{n+1}}
=D^2\phi_{i_{n+1}}=0\,.
$$
Using this and Proposition \ref{Prp: P(G,A) is a ring}$(b)$
(with the pair $(n+1,2)$ instead of $(m,n)$), we obtain
$$
D^{n+2}(Q_{n+1})=D^{(n+1)+2-1}(Q_n\cdot x_{i_{n+1}})=0\,,
$$
which concludes the induction step and proves $(a)$.
\smallskip

\noindent $(b)$ The proof of the first equality in
(\ref{Dn of n-monomial}) is by induction on $n$ (the coincidence of the two
sums in (\ref{Dn of n-monomial})
under the condition $x_{i_1}\cdots x_{i_n}=x_1^{j_1}\cdots x_r^{j_r}$
is an elementary combinatorial fact; we skip its proof).
\smallskip

\noindent First, if $n=1$ we have $Q_1=x_i$ for some
$i\in\{1,...,r\}$, and hence $[D^1 Q_1](g_1)=(x_i+g_{1,i})-x_i=g_{1,i}$.
On the other hand, it is easily seen that the first sum in
(\ref{Dn of n-monomial}) is equal to $g_{1,i}$ whenever
$n=1$, $j_i=1$, and $j_k=0$ for $k\ne i$.
This proves formula (\ref{Dn of n-monomial}) for $n=1$.
\smallskip

\noindent Suppose that for some $n\ge 2$, the formula
(\ref{Dn of n-monomial}) is already proven for all $m\le n-1$.
Let us show that then (\ref{Dn of n-monomial}) holds true
for $m=n$. Take a monomial $Q_n=x_{i_1}\cdots x_{i_n}$ \
($1\le i_1,...,i_n\le r$) and write it in the form
$Q_n=x_{i_1}\cdot Q_{n-1}$, where $Q_{n-1}=x_{i_2}\cdots x_{i_n}$.
By Lemma \ref{Lm: [D{n+1}a](g1,...,g{n+1})=
[Dnag1](g2,...,g{n+1}) - [Dna](g2,...,g{n+1})},
we have
\begin{equation}\label{[Dn Qn}
\aligned
\lbrack D^n Q_n\rbrack (g_1,g_2,...,g_n)
&=[D^{n-1} (Q_n^{g_1})](g_2,...,g_n) - [D^{n-1} Q_n](g_2,...,g_n)\\
&=\{D^{n-1} \left([D^1 Q_n](g_1)\right)\}(g_2,...,g_n)\\
&=\{D^{n-1}\left([D^1 (x_{i_1}\cdot Q_{n-1})](g_1)\right)\}(g_2,...,g_n)\,.\\
\endaligned
\end{equation}
By the Leibniz-type formula (Proposition \ref{Prp: P(G,A) is a ring}$(a)$),
$$
\aligned
\lbrack &D^1 (x_{i_1}\cdot Q_{n-1})\rbrack (g_1)
=(x_{i_1}^{g_1})\cdot [D^1 Q_{n-1}](g_1)
+ \{[D^1 x_{i_1}](g_1)\}\cdot Q_{n-1}\\
&=(x_{i_1}+g_{1,i_1})\cdot [D^1 Q_{n-1}](g_1)+ g_{1,i_1}\cdot Q_{n-1}\\
&=(x_{i_1}+g_{1,i_1})\cdot
[(x_{i_2}+g_{1,i_2})\cdots (x_{i_n}+g_{1,i_n}) - x_{i_2}\cdots x_{i_n}]
+g_{1,i_1}\cdot x_{i_2}\cdots x_{i_n}\,,
\endaligned
$$
which shows that
\begin{equation}\label{D1(xi1 Q{n-1})}
\aligned
\lbrack D^1 (x_{i_1}\cdot Q_{n-1})\rbrack (g_1)
&=x_{i_1}\cdot
\sum_{s=2}^n g_{1,i_s}\cdot x_{i_2}\cdots{\widehat{x_{i_s}}}\cdots x_{i_n}
+g_{1,i_1}\cdot x_{i_2}\cdots x_{i_n}\\
&\hskip100pt+ P_{n-2}(x_1,...,x_r;\,g_1)\\
&=\sum_{s=1}^n g_{1,i_s}\cdot x_{i_1}\cdots{\widehat{x_{i_s}}}\cdots x_{i_n}
+ P_{n-2}(x_1,...,x_r;\,g_1)\,,
\endaligned
\end{equation}
where $P_{n-2}(x_1,...,x_r)$ is a polynomial in $x_1,...,x_r$
of degree at most $n-2$ (with constant coefficiens depending on
the element $g_1\in G$). It follows from statement $(a)$ that
$D^{n-1}P_{n-2}=0$; therefore, formulae (\ref{[Dn Qn})
and (\ref {D1(xi1 Q{n-1})}) imply
\begin{equation}\label{D{n-1}[D1(xi1 Q{n-1})]}
\aligned
\lbrack D^n Q_n\rbrack (g_1,g_2,...,g_n)
&=\Big[ D^{n-1}\sum_{s=1}^n
g_{1,i_s}\cdot x_{i_1}\cdots{\widehat{x_{i_s}}}\cdots x_{i_n}
\Big] (g_2,...,g_n)\\
&=\sum_{s=1}^n
g_{1,i_s}\cdot
\big[D^{n-1} \big(x_{i_1}\cdots{\widehat{x_{i_s}}}\cdots x_{i_n}\big)\big]
(g_2,...,g_n)\,.
\endaligned
\end{equation}
By the induction hypothesis,
\begin{equation}\label{D{n-1}(x{i1}...skip{x{is}}...x{in})}
\big[ D^{n-1}\big(x_{i_1}\cdots {\widehat{x_{i_s}}}
\cdots x_{i_n}\big)\big] (g_2,...,g_n)
=\sum_{{\boldsymbol\sigma}\!_s}
g_{2,i_{{\boldsymbol\sigma}\!_s (2)}}\cdots
g_{n,i_{{\boldsymbol\sigma}\!_s (n)}}\,,
\end{equation}
where ${\boldsymbol\sigma}\!_s$ runs over all one-to-one
mappings $\{2,...,n\}\to\{1,...,{\widehat{s}},...,n\}$.
It follows from (\ref{D{n-1}[D1(xi1 Q{n-1})]}) and
(\ref{D{n-1}(x{i1}...skip{x{is}}...x{in})}) that
$$
\aligned
\lbrack D^n Q_n\rbrack (g_1,g_2,...,g_n)
&=\sum_{s=1}^n g_{1,i_s}\cdot
\big[D^{n-1} \big(x_{i_1}\cdots{\widehat{x_{i_s}}}\cdots x_{i_n}\big)\big]
(g_2,...,g_n)\\
&=\sum_{s=1}^n g_{1,i_s}\cdot
\sum_{{\boldsymbol\sigma}\!_s}
g_{2,i_{{\boldsymbol\sigma}\!_s (2)}}\cdots
g_{n,i_{{\boldsymbol\sigma}\!_s (n)}}\\
&=\sum_{{\boldsymbol\sigma}\in{\mathbf S}(n)}
g_{1,i_{s_1}}g_{2,i_{s_2}}\cdots g_{n,i_{s_n}}\,,
\endaligned
$$
which concludes the induction step and proves Lemma.
\ep

\begin{proposition}\label {Gamma-periodic polynomials are
polynomial-like elements}
Let a lattice $G\subset\mathbb R^r$ of rank $r$
act naturally in $C({\mathbb R}^r)$. A function $p\in C({\mathbb R}^r)$
is a $G$-periodic polynomial of degree at most $n$ if and only if
$D^{n+1}p=0$. In other words,
\begin{equation*}\label {PnGamma is contained in Pn(Gamma,C(Rr))=ker D{n+1}}
P_n^G={\mathcal P}_n(G,C({\mathbb R}^r))
\Def\ker\left[D^{n+1}\colon C({\mathbb R}^r)
\to {\mathcal C}^{n+1}(G,C({\mathbb R}^r))\right]\,.
\end{equation*}
\end{proposition}

\pf
Since all the iterated difference operators are linear,
the inclusion $P_n^G \subseteq \ker D^{n+1}$
follows immediately from Lemma \ref{Lm: Dn of n-monomial}$(a)$.
\medskip

\noindent The proof of the opposite inclusion
$\ker D^{n+1}\subseteq P_n^G$ is by induction in $n$.
Let $p\in C({\mathbb R}^r)$ and $D^1 p=0$; that is,
$p(x+g)=p(x)$ for all $x\in{\mathbb R}^r$ and $g\in G$,
which means that $p$ is a $G$-periodic polynomial of degree $0$.
This gives us the base of induction. Suppose that the inclusion
$\ker D^{n}\subseteq P_{n-1}^G$ is already proven
for some $n\ge 1$. Let $p\in C({\mathbb R}^r)$ and $D^{n+1}p=0$.
By part $(a)$ of Main Theorem, $D^n p$ is a symmetric $n$-polymorphism
on $G$ with values in $C^G({\mathbb R}^r)$. Hence,
$D^n p$ may be recovered from its values
$(D^n p)({\mathbf e}_{i_1},...,{\mathbf e}_{i_n})$, where
${\mathbf e}_1,...,{\mathbf e}_r$ is a free basis of $G$
and $(i_1,...,i_n)$ runs over $\{1,...,r\}^n$.
For every $\boldsymbol\nu=(\nu_1,...,\nu_r)\in{\mathbb Z}_+^r$
with $|\boldsymbol\nu|=\nu_1+...+\nu_r=n$, we
denote by ${\mathbf I}_{\boldsymbol\nu}$ the set of all
${\mathbf i}=(i_1,...,i_n)\in\{1,...,r\}^n$
such that the number of appearences of
each $j\in\{1,...,r\}$ in the sequence ${\mathbf i}=(i_1,...,i_n)$
is precisely $\nu_j$. Since $D^n p$ is symmetric,
the value $(D^n p)({\mathbf e}_{i_1},...,{\mathbf e}_{i_n})$
depends only on $\boldsymbol\nu=(\nu_1,...,\nu_r)\in{\mathbb Z}_+^r$
and does not depend on a choice of a sequence
${\mathbf i}=(i_1,...,i_n)\in{\mathbf I}_{\boldsymbol\nu}$.
Thus, we may define
\begin{equation*}\label{eq: recovering of main part1}
a_{\boldsymbol\nu}
\Def\frac{1}{\boldsymbol\nu\mathbf !}
\cdot(D^n p)({\mathbf e}_{i_1},...,{\mathbf e}_{i_n})\,,
\end{equation*}
where $(i_1,...,i_n)$ is an arbitrary element
of ${\mathbf I}_{\boldsymbol\nu}$.
Clearly $a_{\boldsymbol\nu}$
is a continuous $G$-invariant function on ${\mathbb R}^r$.
Let
\begin{equation}\label{eq: recovering of main part2}
p'=p-\sum_{\boldsymbol\nu: \ |\boldsymbol\nu|=n}
a_{\boldsymbol\nu}(x)\,x_1^{\nu_1}\cdots x_r^{\nu_r}\,.
\end{equation}
Certainly $D^{n+1} p'=0$. Moreover,
$(D^n p')({\mathbf e}_{i_1},...,{\mathbf e}_{i_n})=0$
for all $(i_1,...,i_n)\in\{1,...,r\}^n$ due to Lemma \ref{Lm: Dn of n-monomial}(b).
By part $(a)$ of Main Theorem, this implies that
$D^n p'=0$. The induction hypothesis implies that
$p'$ is a $G$-periodic polynomial of degree at most $n-1$.
According to (\ref{eq: recovering of main part2}), $p$
is a $G$-periodic polynomial of degree at most $n$, which completes
the proof.
\ep

\subsection{$G$-periodic polynomials in $G$-invariant subspaces}
\label{subsect: Gamma-periodic polynomials in Gamma-invariant subspaces}

Let $A$ be a $G$-invariant vector subspace of
$C(\mathbb R^r)$. Set $A^G=A\cap C^G$ and
$P_n^G(A)=A\cap P_n^G$. In other words,
$A^\Gamma$ consists of all $G$-invariant continuous functions that
belong to $A$, and $P_n^G(A)$ consists of all
$G$-periodic polynomials of degree at most $n$ that belong to $A$.
Clearly, $A^G$ and $P_n^G(A)$ are vector spaces.

\begin{theorem}\label {P_nGamma(A) is of finite dimension}
Suppose that the space $A^G$ is of finite dimension.
Then every $P_n^G(A)$ is of finite dimension as well.
\end{theorem}

\pf
By the above definition and Proposition \ref{Gamma-periodic polynomials are
polynomial-like elements}, we have
\begin{equation}\label{P_nGamma(A) is contained in intersect A and
Pn(Gamma,C(Rr))}
P_n^G(A)=A\cap P_n^G=
A\cap {\mathcal P}_n(G,C({\mathbb R}^r))\,.
\end{equation}
On the other hand, it is clear that
\begin{equation}\label{Gamma-polynomials in A}
\aligned
A\cap {\mathcal P}_n(G,C({\mathbb R}^r))
&=A\cap \ker\,\{D^{n+1}\colon C({\mathbb R}^r)
\to{\mathcal C}^{n+1}(G, C({\mathbb R}^r))\}\\
&=\ker\,\{\left.D^{n+1}\right |_A \colon A
\to{\mathcal C}^{n+1}(G, C({\mathbb R}^r))\}\\
&=\ker\,\{\left.D^{n+1}\right |_A \colon A
\to{\mathcal C}^{n+1}(G, A)\}
={\mathcal P}_n(G,A)\,.
\endaligned
\end{equation}
Combining (\ref{P_nGamma(A) is contained in intersect A and Pn(Gamma,C(Rr))})
and (\ref{Gamma-polynomials in A}), we see that
\begin{equation}\label{P_nGamma(A) is contained in Pn(Gamma,A)}
P_n^G(A)={\mathcal P}_n(G,A)\,.
\end{equation}
The lattice $G$ is finitely generated and,
by our assumption, $\dim A^G<\infty$.
Hence, by Main Theorem, $\dim {\mathcal P}_n(G,A)<\infty$, and
(\ref{P_nGamma(A) is contained in Pn(Gamma,A)}) implies
$\dim P_n^G(A)<\infty$.
\ep

\begin{remark}\label{Rmk: we don't assume finite dimension of
coefficients} Any $G$-periodic polynomial $a\in A$
is a sum of monomials with $G$-invariant coefficients.
In the case we know that the coefficients of all these monomials are in $A$
(and thereby, actually, in $A^G$), we could prove
that $\dim P_n^G (A)<\infty$
without referring to Main Theorem.
\smallskip

\noindent Indeed, let us denote by $P_n$ the vector space of all polynomials
in $x_1,...,x_r$ of degree at most $n$
with constant coefficients. The tensor product
${\mathcal T}_n=A^G\otimes P_n$ of the finite dimensional vector spaces
$A^G$ and $P_n$ is of finite dimension.
In fact, ${\mathcal T}_n$ may be represented
as the space of all functions $F(y,x)$ on the direct product
${\mathbb R}^r_y\times{\mathbb R}^r_x$ of the form
$$
\sum_{j_1+...+j_r\le n} f_{j_1,...,j_r}(y_1,...,y_r)\,
x_1^{j_1}\ldots x_r^{j_r}
$$
with coefficients $f_{j_1,...,j_r}\in A^G$.
Any $G$-polynomial $a\in P_n^G (A)$ with
coefficients in $A^G$ may be considered as the restriction
of a certain function $F\in {\mathcal T}_n$ to the diagonal
$\Delta=\{x=y\}$ of ${\mathbb R}^r_y\times{\mathbb R}^r_x$.
Since ${\mathcal T}_n$ is of finite dimension, $P_n^G (A)$
is such as well.
\smallskip

\noindent However, the coefficients of a $G$-periodic
polynomial $a\in P_n^G (A)$ may not be in $A^G$,
and the above ``proof" does not apply in this situation.
\hfill $\bigcirc$
\end{remark}

As above, let $G$ be a full rank lattice in $\mathbb R^r$ and
$\mathcal D$ be a linear partial differential operator in $\mathbb R^r$
with continuous $G$-periodic coefficients.
Let ${\mathcal S}={\mathcal S}_{\mathcal D}$
denote the space of all classical global solutions $u$ of the equation
${\mathcal D}u=0$. Clearly, ${\mathcal S}$ is a $G$-invariant vector
subspace of $C({\mathbb R}^r)$.
Denote by $P_n^G({\mathcal S})$ the space of all solutions $p\in{\mathcal S}$
that are $G$-polynomials of degree at most $n$:
$$
\aligned
P_n^G({\mathcal S})=\Big\{p=\sum_{j_1+...+j_r\le n} f_{j_1,...,j_r}(x_1,...,&x_r)\,
x_1^{j_1}\ldots x_r^{j_r}\,| \\
&\text{all} \ \ f_{j_1,...,j_r}\ \
\text{are} \ \, G\text{-periodic}\,, \ \ {\mathcal D}p=0\Big\}\,.
\endaligned
$$

\noindent The following result follows immediately from
Theorem \ref{P_nGamma(A) is of finite dimension}:
\medskip

\begin{corollary}\label{polynomial-like-solutions}
Suppose that the space ${\mathcal S}^G$
of all $G$-periodic solutions of the equation ${\mathcal D}u=0$
is of finite dimension. Then $\dim P_n^G({\mathcal S})<\infty$
for every $n\in{\mathbb Z}_+ $.
\end{corollary}

\noindent Notice that no additional restrictions to the linear partial
differential operator ${\mathcal D}$ are required.
One has just to assume that {\sl the coefficients of ${\mathcal D}$
{\rm (real or complex)} are continuous and $G$-periodic,
and the space ${\mathcal S}^G$ of all classical $G$-periodic solutions
of the equation ${\mathcal D}u=0$ is of finite dimension.}
Furthermore, the continuity of the coefficients of $\mathcal{D}$ does not seem necessary.
In this case, one can define and then obtain analogous results for certain classes of ``generalized $G$-periodic polynomial" \ so\-lu\-ti\-ons.
However, for general linear partial differential operators $\mathcal D$,
the {\em apriori Liouville-type assumption} $\dim {\mathcal S}^G<\infty$
cannot be omitted, unless ${\mathcal D}$ satisfies an appropriate
maximum principle.
\begin{example}
\label{pdo-real-coefficients}
Let $\mathcal{D}$ be an elliptic operator of second-order with real coefficients acting on functions $u \in C^2(\mathbb{R}^r)$:
$$\mathcal{D}=-\sum_{i,j=1}^r a_{ij}(x)\partial_i \partial_j+\sum_{i=1}^r b_i(x)\partial_i+c(x).$$
Here the coefficients $a_{ij}, b_i, c$ are real, locally H\"older continuous, $\mathbb{Z}^r$-periodic functions. The matrix $A(x):=(a_{ij}(x))$ is positive definite. Also, we assume that the zeroth-order coefficient $c(x)=\mathcal{D}(\textbf{1})$ is non-negative for each $x \in \mathbb{R}^r$, where $\textbf{1}$ is the constant function with value $1$. Then $\mathcal{D}$ satisfies the strong maximum principle (see e.g., \cite[Lemma 3.6]{LinPinchover}).
\end{example}
\subsection{Polynomial-like solutions of periodic differential operators on co-compact Riemannian coverings}
\label{polynomial-like-solutions}
In this subsection, we provide briefly some details as in Subsection \ref{subsect: Gamma-periodic polynomials in Gamma-invariant subspaces} for the case when $\mathcal{D}$ is a periodic differential operator defined on a co-compact Riemannian covering. Let $X$ be a connected Riemannian manifold equipped with an isometric, free, properly discontinuous and co-compact right group action of a finitely generated discrete group $G$ ($G$ may be non-abelian) and let $\mathcal{D}$ be a $G$-periodic elliptic differential operator on $X$, i.e., $\mathcal{D}$ commutes with the group action of $G$. We always assume that the principal symbol of $\mathcal{D}$ is a negative-definite quadratic form. To study $G$-periodic polynomials in this setting, we define the class of additive functions on the covering $X$ as follows (see more details in \cite{Kha, LinPinchover}):
\begin{definition}
\label{add}
A real continuous function $u$ on $X$ is said to be \textit{additive} if there is a homomorphism $\alpha: G \rightarrow \mathbb{R}$ such that
\begin{equation}
\label{additivity}
u^g(x)=u(x)+\alpha(g), \quad \mbox{for all} \quad (g,x) \in G \times X,
\end{equation}
where $u^g(x)=u(g \cdot x)$.

We also denote by $\mathcal{A}(X)$ the vector space consisting of all additive functions on $X$.
%
\end{definition}
It is known that the vector space $\mathcal{A}(X) / C^G(X)$ is isomorphic to $\Hom(G,\mathbb{R})$ (see \cite[Lemma 2.7]{LinPinchover}). Clearly, $\Hom(G,\mathbb{R})=\Hom(\widetilde G,\mathbb{R})$, where $\widetilde G=G/[G,G]$. Hence, the dimension of $\mathcal{A}(X) / C^G(X)$ is equal to the rank $r$ of the finitely generated abelian group $\widetilde G$. Let $\alpha_1, \ldots, \alpha_r$ be a vector basis of $\Hom(\widetilde G, \mathbb{R})\cong \mathbb{R}^r$ and $h_1, \ldots, h_r$ be a corresponding basis of $\mathcal{A}(X)$ (modulo $G$-periodic functions) via the isomorphism between $\mathcal{A}(X) / C^G(X)$ and $\Hom(\widetilde G,\mathbb{R})$. Notice that when $G=\mathbb{Z}^r$ and $X= \mathbb{R}^r$, it is easy to see that $h_j(x)=\phi_j(x)=x_j$ for any $1\leq j \leq r$ and $x \in X$; thus, we may regard these functions $h_1, \ldots, h_r$ as some analogs of Euclidean coordinate functions on the covering $X$ (see also \cite{Agmon} for the case of co-compact abelian coverings). By misuse of language, we say that a \textit{$G$-periodic monomial of degree $n$} is an element $Q_n$ if it has the form $Q_n=f(x)\cdot h_1(x)^{j_1}\ldots h_r(x)^{j_r}$, where the coefficient $f \neq 0$ is $G$-periodic and $j_1, \ldots, j_r \in \mathbb{Z}_+$ such that $j_1+\ldots+j_r=n$. As before, a \textit{$G$-periodic polynomial} is a sum of $G$-periodic monomials and this representation is unique up to a $G$-periodic function. Let $P^G$ ($P^G_n$) be the algebra of $G$-periodic polynomials (of order at most $n$). Then $P^G$ ($P^G_n$) is a $G$-invariant subalgebra (resp. subspace) of $C(X)$.

\noindent The $G$-action on $X$ induces the iterated difference operators
$$D^n: C(X) \rightarrow \mathcal{C}^n(G, C(X)).$$
Again, the subspaces $\mathcal{P}_n(G, C(X))$ of $C(X)$ of \textit{polynomial-like elements in $C(X)$} of order at most $n$ is the kernel of the operator $D^{n+1}$.
Due to (\ref{additivity}), each term $[D^1 h_i](g)(x)=h_i(g\cdot x)-h_i(x)$ is independent of $x$ ($1 \leq i \leq r$). Using this fact, we can repeat the proof of Lemma \ref{Lm: Dn of n-monomial}(a) to see that the same statement should hold, i.e., $D^{n+1}Q$=0 for any $G$-periodic monomial $Q$ of degree at most $n$. This means that
\begin{equation*}
\label{inclusion-kernel}
P^G_n \subseteq \mathcal{P}_n(G, C(X)).
\end{equation*}
Now suppose that $A$ is a $G$-invariant vector subspace of $C(X)$. We also denote $A^G=A \cap C^G(X)$ and $P^G_n(A)=A \cap P^G_n$.
Therefore, $P^G_n(A) \subseteq A \cap \mathcal{P}_n(G, C(X))=\mathcal{P}_n(G,A)$ (see the proof of Theorem \ref{P_nGamma(A) is of finite dimension}). By applying the Main Theorem again, whenever $\dim A^G<\infty$, we have $\dim P^G_n(A)\leq \dim \mathcal{P}_n(G,A)<\infty$ for any $n \in \mathbb{Z}_+$.

It is worthy mentioning that when $G$ is abelian, all of the results in Subsection \ref{subsec: G-periodic polynomials} and Subsection \ref{subsect: Gamma-periodic polynomials in Gamma-invariant subspaces} still hold. The proofs of these results do not require any change in this case, so we skip the details.

We finish this subsection by proving the following statement:
\begin{proposition}
\label{laplace-beltrami-covering}
Let $X$ be a Riemannian manifold which is a Galois covering of a compact Riemannian manifold and $G$ be its deck transformation group. Suppose that the abelianization $\widetilde{G}$ of $G$ has rank $r$ and $G$ is of polynomial growth \footnote{Due to the celebrated work of M. Gromov, this is equivalent to the assumption that $G$ is virtually nilpotent.}. Let $h_1, \ldots, h_r$ be a basis of the vector space $\mathcal{A}(X)$ (modulo $G$-periodic functions on $X$).
Let $\mathcal{D}$ be a $G$-periodic, real elliptic operator of second-order on $X$ such that $\mathcal{D}(\textbf{1}) \geq 0$. Also, let $P_n^G(\mathcal{S}_{\mathcal{D}})$ be the space of all solutions $u$ of the equation $\mathcal{D}u=0$ on $X$ such that
$$u(x)=\sum_{j_1+...+j_r\le n} f_{j_1,...,j_r}(x)\,
h_1(x)^{j_1}\ldots h_r(x)^{j_r},$$
where each term $f_{j_1, \ldots, j_r}(x)$ in the above sum is $G$-periodic.

Then $\dim P_n^G(\mathcal{S}_{\mathcal{D}})<\infty$ for every $n\in{\mathbb Z}_+ $. Furthermore, for any $n \geq 0$, we have
\begin{enumerate}[(i)]
\item
If $\mathcal{D}(\textbf{1}) \neq 0$, $P_n^G(\mathcal{S}_{\mathcal{D}})=\{0\}$.
\item
If $\mathcal{D}(\textbf{1})=0$, the following estimate holds:
$$
\dim P_n^G(\mathcal{S}_{\mathcal{D}}) \leq \binom{n+r}{r}.
$$
\end{enumerate}
\end{proposition}

\begin{proof}
According to our above discussion, $\dim P_n^G(\mathcal{S}_{\mathcal{D}})<\infty$ for all $n \in \mathbb{Z}_+$ if and only if $\dim P_0^G(\mathcal{S}_{\mathcal{D}})<\infty$. It is known (see \cite[Theorem 6.9]{LinPinchover}) that when $\mathcal{D}(\textbf{1})=0$, the dimension of the space $P_0^G(\mathcal{S}_{\mathcal{D}})$ consisting of all $G$-periodic (bounded) solutions on $X$ is one. When $\mathcal{D}(\textbf{1})\neq 0$, \cite[Theorem 4.5]{LinPinchover} yields that $\dim P_0^G(\mathcal{S}_{\mathcal{D}})=0$. In both cases, $P_0^G(\mathcal{S}_{\mathcal{D}})$ has finite dimension.
This proves the first statement. The second statement then follows immediately from the fact that $\dim P_n^G(\mathcal{S}_{\mathcal{D}}) \leq \dim \mathcal{P}_n(G, \mathcal{S}_{\mathcal{D}})$ and Proposition \ref{estimate-dim-Pn}.
\end{proof}

\begin{remark}
Note that in the case $\mathcal{D}(\textbf{1})\neq 0$, Proposition \ref{laplace-beltrami-covering} is still valid even if the growth of $G$ is not polynomial.
\end{remark}

\end{document}